\documentclass{amsart}
\usepackage{graphicx}
\usepackage[T1]{fontenc}
\usepackage{color}
\usepackage[colorlinks,linkcolor=blue,citecolor=blue, pdfstartview=FitH]{hyperref}
\usepackage{hyperref}
\usepackage{amssymb}
\usepackage{epsf}
\usepackage{caption}
\usepackage{tikz,tikz-cd}
\usepackage{mathtools} 
\mathtoolsset{showonlyrefs}

\def\p{\partial}

\def\mb{\mathbb}
\def\mc{\mathcal}

\def\n{\nabla}

\theoremstyle{plain}

\theoremstyle{plain}
\newtheorem{theorem}{Theorem}[section]
\theoremstyle{plain}
\newtheorem{definition}[theorem]{Definition}
\theoremstyle{plain}
\newtheorem{lemma}[theorem]{Lemma}
\theoremstyle{remark}
\newtheorem{remark}[theorem]{Remark}
\theoremstyle{remark}

\theoremstyle{remark}

\theoremstyle{plain}
\newtheorem{proposition}[theorem]{Proposition}
\numberwithin{equation}{section}
\theoremstyle{plain}

\usepackage{fancyhdr}
\makeatletter
\@namedef{subjclassname@2020}{\textup{2020} Mathematics Subject Classification}
\makeatother

\begin{document}

\title{Spectral constant rigidity of warped product metrics}

\author{Xiaoxiang Chai}
\address{Xiaoxiang Chai: Department of Mathematics, POSTECH, Pohang, Gyeongbuk, South Korea}
\email{xxchai@kias.re.kr, xxchai@postech.ac.kr}

\author{Juncheol Pyo}
\address{Juncheol Pyo: Department of Mathematics, Pusan National University, Busan 46241, Korea and
Korea Institute for Advanced Study, Seoul 02455, Korea}
\email{jcpyo@pusan.ac.kr}

\author{Xueyuan Wan}
\address{Xueyuan Wan: Mathematical Science Research Center, Chongqing University of Technology, Chongqing 400054, China.}
\email{xwan@cqut.edu.cn}


\begin{abstract}
  A theorem of Llarull says that if a smooth metric $g$ on the
  $n$-sphere $\mathbb{S}^n$ is bounded below by the standard round metric and
  the scalar curvature $R_g$ of $g$ is bounded below by $n (n - 1)$, then the metric $g$ must be the standard round metric. We prove
  a spectral Llarull theorem by replacing the bound $R_g \geq n (n - 1)$
  by a lower bound on the first eigenvalue of an elliptic operator involving
  the Laplacian and the scalar curvature $R_g$. We utilize two methods:
 spinor and spacetime harmonic function.
\end{abstract}

  \thanks{X. Chai has been supported by the National Research Foundation of Korea (NRF) grant
funded by the Korea government (MSIT) (No. RS-2024-00337418) and an NRF grant No. 2022R1C1C1013511. J. Pyo is partially supported by NRF grant funded by MSIT (No. NRF-2020R1A2C1A01005698 and NRF-2021R1A4A1032418). X. Wan is partially supported by the National Natural Science Foundation of China (Grant No. 12101093) and the Natural Science Foundation of Chongqing (Grant No. CSTB2022NSCQ-JQX0008), the Scientific Research Foundation of the Chongqing University of Technology. X. Chai would like to thank Yukai Sun (PKU, Beijing) for pointing out a typo in the constants $a,b$. }
\maketitle


\section{Introduction}

In his \textit{Four lectures on scalar curvature} \cite[Section 6.1.2]{gromov-four-2021}, Gromov asked the following question:

\textit{  What are effects on the topology and/or metric geometry of a Riemannian manifold $X$ played by the positivity of the}
\[L_\gamma : f(x) \mapsto -\Delta f(x) + \gamma \cdot \mathrm{Sc}(X,x)f(x)\]
\textit{for a given constant $\gamma>0$?}

In \cite{hirsch-spectral-2023-arxiv}, a spectral version of the band width inequality (\cite{gromov-metric-2018}) was proven and it is similar to a conjecture in \cite[Section 6.1.2]{gromov-four-2021} (listed as Item 7 in that section).
We are interested in a spectral version of the following scalar curvature comparison on $\mathbb{S}^n$ of Llarull {\cite{llarull-sharp-1998}}.

\begin{theorem}
  \label{llarull}If $(\mathbb{S}^n, g)$ satisfies $g \geq \bar{g}$ and
  $R_g \geq n (n - 1)$ where $\bar{g}$ is the standard round metric on
  $\mathbb{S}^n$, then $g = \bar{g}$.
\end{theorem}

Currently there are various proofs available for this result, for instance, by spinor methods \cite{llarull-sharp-1998, BBHW, WX}, spacetime harmonic function \cite{hirsch-rigid-2022-arxiv} and by surfaces of prescribed mean curvature \cite{hu-rigidity-2023}. 
The works \cite{BBHW, WX} also handled warped product metrics.

Now we introduce the concept of a $c$-spectral constant introduced in {\cite{hirsch-spectral-2023-arxiv}}.
\begin{definition}[$c$-spectral constant]
Let $( M^n ,g)$ be an $n$-dimensional open Riemannian manifold. For any $c\in \mathbb{R}$,  the {\it $c$-spectral constant} is defined as 
\begin{equation*}
  \Lambda_c=\inf\left\{\int_ M (|\nabla u|^2+cR_gu^2)dV\Big|u\in H_0^1( M ),\int_ M  u^2dV=1\right\},\label{L c}
\end{equation*}
where $R_g$ denotes the scalar curvature of $g$ and $H_0^1\left(M^n\right)$ is the Sobolev space of $L^2$ functions with square integrable derivatives arising as the completion of $C_0^{\infty}\left(M^n\right)$, the space of smooth functions with compact support, in the Sobolev $H^1$-norm. When $M^n$ is a compact manifold with boundary, $\Lambda_c$ is defined as the $c$-spectral constant of the interior $\mathring{M}^n$.
\end{definition}
When it was needed to emphasize the dependence of $\Lambda_c$ on the metric $g$, we also use the notation $\Lambda_c(g)=\Lambda_c$.
We replace the pointwise scalar curvature bound $R_g \geq n (n - 1)$ by
the bound
\begin{equation}
  \Lambda_c \geq \Lambda > 0, \label{L c lower bound}
\end{equation}
which can be viewed as a bound on the scalar curvature $R_g$ in some integral sense or weak
sense.

Before stating our spectral version of Llarull's result \cite{llarull-sharp-1998}, we give a description of the model metric.
We fix two positive constants $c>\tfrac{1}{4}$ and $\Lambda$. Let $(\mathbb{S}^{n-1},g_{\mathbb{S}^{n-1}})$ be the standard round sphere of dimension $n-1$ and $g_0$ be  the following warped product metric 
\begin{equation*}
  g_0= d\theta\otimes d\theta+a^2\sin^2(b\theta)g_{\mathbb{S}^{n-1}}
\end{equation*}
on $ (0,\frac{\pi}{b})\times\mathbb{S}^{n-1}$, where $a,b>0$ are two positive constants given by
\begin{align}\begin{split}
  a&=\sqrt{\frac{(n-1)(n-2)(4n-(n-1) c^{-1} )}{\Lambda(4(n-2)-(n-3) c^{-1} )}},\\
  b&=\frac{\sqrt{\Lambda}(4- c^{-1} )}{\sqrt{(4n-(n-1) c^{-1} )(4(n-1)-(n-2) c^{-1} )}}.
  \end{split}\end{align}
From Appendix \ref{App}, the two constants $a,b$ satisfy the following condition
\begin{equation}
  \beta_2 a_1^2-\beta_2 a_1 b-\frac{n(n-2)}{4a^2}=0\text{ and }\Lambda_{ c }(g_0)=\Lambda, \label{beta2 constants relation}
  \end{equation}
where 
$
  a_1=\frac{1}{2}\sqrt{\frac{\Lambda(4n-(n-1) c^{-1} )}{4(n-1)c-(n-2)}}$ and $\beta_2=\frac{n}{n-1}\frac{4(n-1)-(n-2) c^{-1} }{4n-(n-1) c^{-1} }.
  $
Also, we remark that our requirement that $c>\tfrac{1}{4}$ is due to the fact that in the proof of Theorem \ref{thm1} we have used the function $u\in H_0^1$ achieving
  $\Lambda_c(g_0)$ in the manifold $(\mathbb{S}^{n-1}\times (0,\tfrac{\pi}{b}),g_0)$ satisfies $\lim_{s\to 0}u(s)=\lim_{s\to \pi/b}u(s)=0$ (see Appendix \ref{App}).

Our spectral Llarull theorem is the following.
\begin{theorem}\label{thm1}
 Let $M$ be an $n$-dimensional, non-compact, connected spin manifold without boundary. Fix two positive constants $c>\frac{1}{4}$ and $\Lambda$, and let $g$ be a Riemannian metric on $ M $ such that
  $\Lambda_c (g) \geq \Lambda_c(g_0)=\Lambda$. Suppose that $\Phi:( M ,g)\to ( (0,\frac{\pi}{b})\times\mathbb{S}^{n-1},g_0)$ is a smooth map with the following properties:
\begin{itemize}
  \item $\Phi$ is proper;
  \item $\Phi$ has non-zero degree;
  \item $\Phi$ is $1$-Lipschitz.
\end{itemize}
 Then $\Phi$ is a Riemannian isometry and ${\Lambda}_ c =\Lambda$. 
\end{theorem}

For $c\neq 0$, denote $ \kappa :=c^{-1}$, and we define the following 
\begin{equation*}
\tilde{\Lambda}_{ \kappa }:=  c^{-1}\Lambda_c=\inf\left\{\int_ M ( c^{-1} |\nabla u|^2+R_gu^2)dV\Big|u\in H_0^1( M ),\text{ }\int_ M  u^2dV=1\right\}.
\end{equation*}
The definition of $\tilde{\Lambda}_{\kappa}$ is easily seen to extend to $ \kappa =0$, and it is given by
\begin{equation*}
\tilde{\Lambda}_{ 0} =\inf\left\{\int_ M  R_gu^2dV \Big| u \in H_0^1( M ),\text{ }\int_ M  u^2dV=1\right\}.
\end{equation*}
A lower bound $\tilde{\Lambda}_0\geq\tilde{\Lambda}$ for some number $\tilde{\Lambda}$ is easily seen to be equivalent to the pointwise bound $R_g\geq \tilde{\Lambda}$.
The proof of Theorem \ref{thm1} carries over to the case $\kappa=0$ with suitable modifications.
In fact, for $ \kappa =0$ and $\tilde{\Lambda}_0\geq \tilde{\Lambda}>0$, then
\begin{equation*}
  a=\sqrt{\frac{n(n-1)}{\tilde{\Lambda}}}=\frac{1}{b}.
\end{equation*}
By rescaling, we may assume that $\tilde{\Lambda}=n(n-1)$, then $R_g\geq n(n-1)$.
In this case, $a=b=1$, and the above theorem implies the following rigidity result.
\begin{theorem}[{\cite[Theorem 1.4]{BBHW} and \cite[Theorem 1.5]{WX}}]
Let $n\geq 3$ and let $( M ,g)$ be an $n$-dimensional non-compact, connected spin manifold without boundary, $R_g\geq n(n-1)$. Suppose that $\Phi:( M ,g)\to ( (0,\pi)\times\mathbb{S}^{n-1},d\theta\otimes d\theta +\sin^2(\theta)g_{\mathbb{S}^{n-1}})$ is a smooth map satisfying the properties in Theorem \ref{thm1}. Then $\Phi$ is a Riemannian isometry.
\end{theorem}

As mentioned before, the Llarull theorem can be generalized to warped product metrics \( {d}t^2+\phi(t)^2 g_{\mathbb{S}^{n-1}}\) with $\phi$ satisfying the condition that $(\log\phi)''$ is strictly negative.
See for instance \cite[Theorem 1.1]{BBHW}. We intend to formulate a similar spectral result so that the limiting case recovers these results as well.
Let $\mu$ be a function on $M$, we define the number $\Lambda_{c,\mu}$ for any open manifold $(M,g)$
\begin{equation*}
  \Lambda_{c,\mu}=\inf\left\{
    \int_ M (|\n u|^2+cR_gu^2-\mu u^2)dV \Big | u\in H_0^1( M ),\int_ M  u^2dV=1
  \right\}.
\end{equation*}
Note that the condition $\Lambda_c \geq \Lambda$ we used before is equivalent to $\Lambda_{c,\Lambda} \geq 0$. In some sense, the condition $\Lambda_{c,\mu}\geq 0$ generalizes the pointwise scalar curvature bound by a function.

Let \(g_1= {d}t^2+\phi(t)^2 g_{\mathbb{S}^{n-1}}\) where the warping factor $\phi$ is a function defined on a closed interval $[0,t_0]$ ($t_0>0$) with the property $\phi(0)=\phi(t_0)=0$,
\[\lim_{t\to 0}\phi'(t) / \phi(t)= -\lim_{t\to t_{0}}{\phi'(t)} / {\phi(t)}=+\infty\]
and $(\log \phi)''<0$. As before, we fix the constant $c>\tfrac{1}{4}$ and $\beta_2$ given in \eqref{beta2 constants relation}.
Let \begin{equation}
  f =\tfrac{4n-(n-1)c^{-1}}{n(4-c^{-1})} \tfrac{\phi'}{\phi} \label{f in mu}
\end{equation}
and $\mu=\mu(t)$ be the function which is given by 
\begin{equation}
  \frac{c^{-1}\mu}{4}+\frac{n(n-1)}{4}\beta_2 f^2+\frac{n-1}{2}\beta_2f'-\frac{(n-2)(n-1)}{4}\phi^{-2}=0. \label{f mu relation}
\end{equation}
See Appendix \ref{App} for the fact that $v=\phi^{\tfrac{2}{4-c^{-1}}}$ satisfies the equation
\[
  -\Delta_{g_1} v + c R_{g_1}v=\mu v.
  \]
  
\begin{theorem}\label{thm1-mu}
  Let $( M ,g)$ be an $n$-dimensional, non-compact, connected spin manifold without boundary. Let $\Phi:( M ,g)\to ( (0,t_0)\times\mathbb{S}^{n-1},g_1)$ be a smooth map with the following properties:
\begin{itemize}
  \item $\Phi$ is proper;
  \item $\Phi$ has non-zero degree;
  \item $\Phi$ is $1$-Lipschitz.
\end{itemize}
Suppose $g$ is a Riemannian metric on $ M $ such that
  $\Lambda_{c,\mu\circ \Phi} (g) \geq \Lambda_{c,\mu}(g_1)=0$.
 Then $\Phi$ is a Riemannian isometry and ${\Lambda}_{c,\mu\circ\Phi}  =0$. 
\end{theorem}

Similar to the case when  $\mu=\Lambda$, we inspect the limiting case as $c\to \infty$. Let $\mu'$ be the scalar
curvature of some warped product $dt^2+\phi(t)^2 g_{\mathbb{S}^{n-1}}$. The lower bound $c^{-1}\Lambda_{c,\mu\circ\Phi}\ge 0$ reduces to the pointwise bound $R_g \ge \mu'\circ\Phi$
in the limit.

Our approach to Theorems \ref{thm1} and \ref{thm1-mu} is spinorial and uses ideas from articles \cite{BBHW,hirsch-spectral-2023-arxiv, WX}. In dimension 3, we are also able to give a proof using spacetime harmonic function which follows \cite{hirsch-spectral-2023-arxiv}. (See Theorem \ref{harmonic main}). The harmonic function was initiated by Stern \cite{stern-scalar-2022} and later generalized to the spacetime settings by \cite{hirsch-spacetime-2022}. It has been proven quite useful.
However, compared to the spinorial proof, other than the restriction in dimension 3, some extra technical conditions are also imposed which we believe are removable. The extra conditions are due to Lemma \ref{lower gradient bound}. 

\

The article is organized as follows:

In Section \ref{spin proof}, we present a proof of our main Theorem \ref{thm1} using spinors.
In Section \ref{sec:harmonic proof}, we prove Theorem \ref{thm1} in dimension 3 (see Theorem \ref{harmonic main}) using an
alternative approach based on spacetime harmonic functions.
At the end of Section \ref{spin proof}, we show that it is relatively
easy to adapt the proofs to obtain the slightly more general spectral rigidity result Theorem \ref{thm1-mu}.
In Appendix \ref{App}, we calculate explicitly the extremal function $u$ achieving the spectral constant $\Lambda:=\Lambda_c(g_0)$ and
constants $a$, $b$ and also verify the case for $\Lambda_{c,\mu}=0$.

\section{Spinor proof of Theorem \ref{thm1}}\label{spin proof}

In this section, we follow the method in \cite{BBHW,hirsch-spectral-2023-arxiv, WX} to prove Theorem \ref{thm1}. 
We introduce the function
\begin{equation}\label{f}
  f(\theta)=\frac{2a_1}{n}\cot(b\theta)
  \end{equation}
critical to our proof.
Since $$\beta_2 a_1^2-\beta_2 a_1 b-\frac{n(n-2)}{4a^2}=0$$ holds, we see that $f$ solves the following ordinary differential equation 
\begin{equation*}
  \frac{c^{-1}\Lambda}{4}+\frac{n(n-1)}{4}\beta_2 f^2+\frac{n-1}{2}\beta_2f'-\frac{(n-2)(n-1)}{4}\frac{1}{a^2\sin^2(b\theta)}=0.
\end{equation*}

Firstly, we consider the case where $n$ is even. 
Following \cite{BBHW}, we denote 
\begin{equation*}
  \Phi=(\varphi,\Theta), \text{ where } \varphi: M \to \mathbb{S}^{n-1},\text{ and }\Theta: M \to (0,\frac{\pi}{b}).
\end{equation*}
By assumption, $\Phi$ is $1$-Lipschitz, one has
\begin{equation*}
  g\geq d\Theta\otimes d\Theta+ a^2\sin^2(b\Theta)\varphi^*g_{\mathbb{S}^{n-1}},
\end{equation*}
from which it follows that 
\begin{equation}\label{eqn13}
  |\n\Theta|^2\geq |\n\Theta|^4+a^2\sin^2(b\Theta)|d\varphi(\n\Theta)|^2_{g_{\mathbb{S}^{n-1}}}\geq |\n\Theta|^4,
\end{equation}
and so $|\n\Theta|\leq 1$, the inequality is strict unless $d\varphi(\n\Theta)=0$, see \cite[Lemma 2.2]{BBHW}.

Let $z\in \mathbb{S}^{n-1}\times [\frac{\pi}{3b},\frac{2\pi}{3b}]$ be a regular value of $\Phi$. Then $\Phi^{-1}(\{z\})$ is a finite subset of $ M $. Fix a real number $\delta_0\in (0,\frac{\pi}{4b})$ such that $\Phi^{-1}(\{z\})$ is contained in a single connected component of $\Phi^{-1}(\mathbb{S}^{n-1}\times [\delta_0,\frac{\pi}{b}-\delta_0])$. Set
\begin{equation*}
  \Delta=\left\{\delta\in(0,\delta_0)\Big|\delta \text{ and }\frac{\pi}{b}-\delta \text{ are regular values of } \Theta: M \to (0,\frac{\pi}{b})\right\}.
\end{equation*}
Then $\Delta$ is an open and dense subset of $(0,\delta_0)$ by Sard's theorem. For any $\delta\in \Delta$, let $M_\delta$ be the connected component of $\Phi^{-1}(\mathbb{S}^{n-1}\times [\delta_0,\frac{\pi}{b}-\delta_0])$ that contains $\Phi^{-1}(\{z\})$. By the definition of $\Delta$ and $\Phi$ is proper, $M_\delta$ is a compact, connected manifold with boundary. Then $\p M_\delta=\p_+M_\delta\cup \p_-M_\delta$, where
\begin{equation*}
  \p_+M_\delta:=\p M_\delta\cap \Phi^{-1}(\mathbb{S}^{n-1}\times \{\frac{\pi}{b}-\delta\})\text{ and } \p_-M_\delta:=\p M_\delta\cap \Phi^{-1}(\mathbb{S}^{n-1}\times \{\delta\}).
\end{equation*}
Since $\Phi^{-1}(\{z\})\subset M_\delta$, the restriction $\Phi|_{M_\delta}\to \mathbb{S}^{n-1}\times [\delta,\frac{\pi}{b}-\delta]$ has the same degree as the map $\Phi: M \to \mathbb{S}^{n-1}\times (0,\frac{\pi}{b})$, see \cite[Lemma 2.2]{BBHW}. 
Set 
\begin{equation*}
  a_1=\frac{1}{2}\sqrt{\frac{\Lambda(4n-(n-1) c^{-1} )}{4c(n-1)-(n-2)  }}.
\end{equation*}
For any $\delta\in (0,\frac{\pi}{4b})$ and $\varepsilon\in (0,\delta)$, there is a smooth function $\psi_{\delta,\varepsilon}:[\delta,\frac{\pi}{b}-\delta]\to \mathbb{R}$ as a perturbation of $f$ introduced earlier in \eqref{f} and a constant $K>0$ (independent of $\delta,\varepsilon$)
satisfying the following properties:
\begin{itemize}
  \item $\psi_{\delta,\varepsilon}(\theta)=\frac{2a_1}{n}\frac{\cos (b(\theta-\delta+\varepsilon))}{\sin(b(\theta-\delta+\varepsilon))}$ for $\theta\in [\delta,\frac{\pi}{3b}]$;
  \item $\psi_{\delta,\varepsilon}(\theta)=\frac{2a_1}{n}\frac{\cos(b(\theta+\delta-\varepsilon))}{\sin(b(\theta+\delta-\varepsilon))}$ for $\theta\in [\frac{2\pi}{3b},\frac{\pi}{b}-\delta]$;
  \item $\left|\psi_{\delta,\varepsilon}(\theta)-\frac{2a_1}{n}\frac{\cos(b\theta)}{\sin(b\theta)}\right|\leq K\delta$ for $\theta\in [\frac{\pi}{3b},\frac{2\pi}{3b}]$;
  \item $\left|\psi'_{\delta,\varepsilon}(\theta)+\frac{2a_1}{n}\frac{b}{\sin^2(b\theta)}\right|\leq K\delta$ for $\theta\in [\frac{\pi}{3b},\frac{2\pi}{3b}]$.
\end{itemize}
One can refer to \cite{BBHW, hirsch-rigid-2022-arxiv} for the existence of such smooth functions. Denote 
$$\Psi_{\delta,\varepsilon}=\psi_{\delta,\varepsilon}\circ \Theta:M_\delta\to \mathbb{R}.$$

Now we consider the following smooth map 
\begin{equation*}
  h:\mathbb{S}^{n-1}\times \mb{S}^1\to \mb{S}^n,\quad h(x,t)=(-\cos(t)x,-\sin(t)),\,\,t\in [0,2\pi].
\end{equation*}
\begin{proposition}\label{prop2.1}
The map $h:(\mathbb{S}^{n-1}\times \mb{S}^1,g_{\mathbb{S}^{n-1}}+dt\otimes dt)\to (\mb{S}^n,g_{\mb{S}^n})$ satisfies:
	\begin{itemize}
  \item $h:\mathbb{S}^{n-1}\times\{\pi\}\to \mb{S}^n$ is an isometric embedding;
  \item If $v\in T_x\mathbb{S}^{n-1}$ and $w\in T_{\pi}\mb{S}^1$, then $h_*(v)\perp h_*(w)$ at the point $h(x,\pi)\in \mb{S}^n$;
  \item Let $\Gamma:\mb{S}^1\to \mb{S}^1$ be the isometry that reverses the orientation and $\Gamma(0)=0$, $\Gamma(\pi)=\pi$, where we have parameterized $\mb{S}^1$ by $[0,2\pi]$. We have the following commuting diagram
  \begin{equation*}
  \begin{tikzcd}
\mathbb{S}^{n-1}\times \mb{S}^1 \arrow[r, "\mathrm{id}\times \Gamma"] \arrow[d, "h"] & \mathbb{S}^{n-1}\times \mb{S}^1 \arrow[d, "h"] \\
\mb{S}^n \arrow[r, "\Gamma_n"]                                       & \mb{S}^n                             
\end{tikzcd}
\end{equation*}
 where $\Gamma_n:\mb{S}^n\to \mb{S}^n$ is the isometry of $\mb{S}^n$ obtained by the restriction to $\mb{S}^n$ of the reflection on $\mathbb{R}^{n+1}$ along the hyperplane $x_{n+1}=0$.
\end{itemize}
\end{proposition}
\begin{proof}
\begin{itemize}
\item	The pull-back metric $h^*g_{\mb{S}^n}$ satisfies 
\begin{align}\begin{split}\label{eqn2}
\begin{split}
    h^*g_{\mb{S}^n}&=\sum_{i=1}^{n}(d(\cos(t)x^i))^2+(d\sin(t))^2\\
    &=(dt)^2+\cos^2(t)\sum_{i=1}^{n}(dx^i)^2\\
    &=(dt)^2+\cos^2(t) \cdot g_{\mathbb{S}^{n-1}}.
 \end{split}
\end{split}\end{align}
In particular, for $t=\pi$, $h^*g_{\mb{S}^n}=g_{\mathbb{S}^{n-1}}$. Since $h(x,\pi)=(x,0)$, $h:\mathbb{S}^{n-1}\times\{\pi\}\to \mb{S}^n$ is an isometric embedding.
\item For any $v\in T_x\mathbb{S}^{n-1}$ and $w\in T_\pi \mb{S}^1$, one has
\begin{align*}
\begin{split}
  g_{\mb{S}^n}(h_*(v),h_*(w))&=h^*g_{\mb{S}^n}(v,w)=((dt)^2+\cos^2(t) \cdot g_{\mathbb{S}^{n-1}})(v,w)=0,
 \end{split}
\end{align*}
from which it follows that $h_\ast(v)\bot h_*(w)$ at the point $h(x,\pi)\in \mathbb{S}^n$.
\item The isometry $\Gamma:\mb{S}^1\to \mb{S}^1$ is given by $\Gamma(t)=2\pi-t$, and $\Gamma_n:\mb{S}^n\to \mb{S}^n$ is given by $\Gamma_n(x,y)=(x,-y)$, where $(x,y)\in \mb{S}^n$ and $x\in \mathbb{R}^n$, $y\in \mathbb{R}$. Hence 
\begin{align*}
\begin{split}
  \Gamma_n\circ h(x,t)&=\Gamma_n(-\cos(t)x,-\sin(t))\\
  &=(-\cos(t)x,\sin(t))\\
  &=h(x,2\pi-t)\\
  &=h\circ (\mathrm{id}\times \Gamma)(x,t).
 \end{split}
\end{align*}
\end{itemize}
The proof is complete.
\end{proof}
\begin{remark}
In \cite[Lemma A.1]{Kra}, W. Kramer proved the existence of the map $h:\mathbb{S}^{n-1}\times \mb{S}^1\to \mb{S}^n$ that satisfies the properties in the above proposition. 
\end{remark}

We consider the product $\tilde{ M }= M \times \mb{S}^1$ equipped with the product metric $\tilde{g}=g+r^2 g_{\mb{S}^1}$, where $r>0$ is a real parameter to be determined, and we will make $r$ sufficiently large later on. Denote $\tilde{M}_\delta=M_\delta\times \mb{S}^1\subset \tilde{ M }$. We write $\p\tilde{M}_\delta=\p_+\tilde{M}_\delta\cup \p_-\tilde{M}_\delta$, where
\begin{equation*}
  \p_+\tilde{M}_\delta=\p_+M_\delta\times \mb{S}^1,\quad \p_-\tilde{M}_\delta=\p_-M_\delta\times \mb{S}^1.
\end{equation*}	
We define a smooth map $\tilde{f}:\tilde{ M }= M \times \mb{S}^1\to \mb{S}^n$ by 
\begin{equation*}
  \tilde{f}(x,t)=h(\varphi(x),t)
\end{equation*}
for any $x\in  M $ and $t\in \mb{S}^1$. 

Given a spin structure for \(  M  \) and let \( S \) denote its associated spinor bundle. For \( \tilde{ M } =  M  \times \mb{S}^1 \), where \( \mb{S}^1 \) is equipped with a trivial spin structure, let \( \tilde{S} \) denote the spinor bundle of $\tilde{ M }$. Thus, \( \tilde{S} \) is  the pull-back of \( S \) through the canonical projection mapping \( \tilde{ M } =  M  \times \mb{S}^1 \) onto \(  M  \).
Let $E_0$ denote the spinor bundle
over the round sphere $\mb{S}^n$, equipped with a natural metric and connection.
Since $n$ is even, we may decompose $E_0$ as $E_0=E_0^+\oplus E_0^-$, where $E_0^+$ and $E_0^-$ are the eigenbundles of the volume form.

For each $\delta\in \Delta$, let $\nu$ be the outward unit normal vector and consider the indices of the following operators:
\begin{itemize}
  \item  Let $\mathrm{ind}_1$ denote the index of the Dirac operator on $\tilde{S} \otimes \tilde{f}^* E_0^{+}$ with boundary conditions $u=-i \nu \cdot u$ on $\partial_{+} \tilde{M}_\delta$ and $u=i \nu \cdot u$ on $\partial_{-} \tilde{M}_\delta$;
\item  Let $\mathrm{ind}_2$ denote the index of the Dirac operator on $\tilde{S} \otimes \tilde{f}^* E_0^{+}$ with boundary conditions $u=i \nu \cdot u$ on $\partial_{+} \tilde{M}_\delta$ and $u=-i \nu \cdot u$ on $\partial_{-} \tilde{M}_\delta$;
\item Let $\mathrm{ind}_3$ denote the index of the Dirac operator on $\tilde{S} \otimes \tilde{f}^* E_0^{-}$ with boundary conditions $u=-i \nu \cdot u$ on $\partial_{+} \tilde{M}_\delta$ and $u=i \nu \cdot u$ on $\partial_{-} \tilde{M}_\delta$;
\item Let $\mathrm{ind}_4$ denote the index of the Dirac operator on $\tilde{S} \otimes \tilde{f}^* E_0^{-}$ with boundary conditions $u=i \nu \cdot u$ on $\partial_{+} \tilde{M}_\delta$ and $u=-i \nu \cdot u$ on $\partial_{-} \tilde{M}_\delta$.
\end{itemize}
 Denote
\begin{equation*}
  \Delta_i=\{\delta\in\Delta:\mathrm{ind}_i>0\},\quad 1\leq i\leq 4.
\end{equation*}
By \cite[Proposition 2.4]{BBHW}, $\Delta=\Delta_1\cup \Delta_2\cup\Delta_3\cup\Delta_4$. Hence, one of the sets \( \Delta_i,1\leq i\leq 4 \) must contain 0 within its closure. By interchanging the bundles \( E_0^{+} \) and \( E_0^{-} \) as necessary, we can infer that either \( \Delta_1 \) or \( \Delta_2 \) contains 0 in its closure. For our subsequent discussions, we assume that \( \Delta_1 \) contains $0$ in its closure. (The case when the set $\Delta_2$ contains $0$ in its closure can be handled analogously.)

In the following, we assume that $\delta\in \Delta_1$. Denote $\tilde{E}=\tilde{f}^*E_0^+$. Let $\n^{\tilde{S}\otimes\tilde{E}}$ denote the natural connection induced from $\tilde{S}$ and $\tilde{E}$. Let $ \mc{D}^{\tilde{S}\otimes\tilde{E}}$ denote the Dirac operator acting on the sections of $\tilde{S}\otimes\tilde{E}$,
\begin{equation*}
  \mc{D}^{\tilde{S}\otimes\tilde{E}}u=\sum_{k=1}^{n+1}e_k\cdot\n^{\tilde{S}\otimes \tilde{E}}_{e_k}u,
\end{equation*}
where $\{e_1,\ldots,e_{n+1}\}$ is a local orthonormal frame on $\tilde{ M }$. Since $\delta\in\Delta_1$, we know that $\mathrm{ind}_1>0$. In view of the deformation invariance of the index, the index of the operator $\mc{D}^{\tilde{S}\otimes\tilde{E}}-\frac{in}{2}\Psi_{\delta,\varepsilon}$ with the same boundary conditions equals to $\mathrm{ind}_1$, and so its index is also positive. Hence the kernel of the operator $\mc{D}^{\tilde{S}\otimes\tilde{E}}-\frac{in}{2}\Psi_{\delta,\varepsilon}$ is not empty, and so
we can find a section $u\in C^\infty(\tilde{M}_\delta,\tilde{S}\otimes \tilde{E})$ such that
\begin{itemize}
  \item $u$ does not vanish identically;
  \item $\mc{D}^{\tilde{S}\otimes\tilde{E}}u-\frac{in}{2}\Psi_{\delta,\varepsilon}u=0$ on $\tilde{M}_\delta$;
  \item $u=-i\nu\cdot u$ on $\p_+\tilde{M}_\delta$ and $u=i\nu\cdot u$ on $\p_-\tilde{M}_\delta$.
\end{itemize}
The spinor bundles $\tilde{S}\otimes \tilde{E}$ and $(\mathrm{id}\times \Gamma)^*(\tilde{S}\otimes \tilde{E})$ can be canonically identified, where $\mathrm{id}\times \Gamma:M_\delta\times \mb{S}^1\to M_\delta\times \mb{S}^1$, see  \cite[Page 599]{Kra} and \cite[Page 24]{WX}. In fact, by the commutative diagram in Proposition \ref{prop2.1}, one has
\begin{equation}
    (\mathrm{id}\times\Gamma)^*(T^*(M_\delta\times \mb{S}^1)\oplus \tilde{f}^*T^*\mathbb{S}^n)=T^*M_\delta\oplus \Gamma^*T^*\mb{S}^1\oplus \tilde{f}^*\Gamma_n^*T^*\mathbb{S}^n,
\end{equation}
It follows that 
$$(\mathrm{id}\times\Gamma)^*=(\mathrm{id},\Gamma^*,\Gamma^*_n):T^*(M_\delta\times \mb{S}^1)\oplus \tilde{f}^*T^*\mathbb{S}^n\to  T^*(M_\delta\times \mb{S}^1)\oplus \tilde{f}^*T^*\mathbb{S}^n$$ 
is an isometry and preserves the orientation of $T^*(M_\delta\times \mb{S}^1)\oplus \tilde{f}^*T^*\mathbb{S}^n$. This map induces an identification between the spinor bundles $\tilde{S}\otimes\tilde{f}^*E_0$ and $(\mathrm{id}\times\Gamma)^*(\tilde{S}\otimes\tilde{f}^*E_0)$ because the spinor bundle associated with the cotangent bundle $T^*(M_\delta\times \mb{S}^1)\oplus \tilde{f}^*T^*\mathbb{S}^n$ is the bundle $\tilde{S}\otimes\tilde{f}^*E_0$. In particular, the two spinor subbundles  $\tilde{S}\otimes\tilde{E}$ and $(\mathrm{id}\times\Gamma)^*(\tilde{S}\otimes\tilde{E})$ can also be identified.

Consider the following section
\begin{equation*}
  \tilde{u}=u+(\mathrm{id}\times\Gamma)^*u\in C^\infty(\tilde{M}_\delta,\tilde{S}\otimes \tilde{E}).
\end{equation*}
Following \cite{Kra,WX}, one can prove that the section $\tilde{u}$ satisfies
\begin{itemize}
  \item $\tilde{u}$ does not vanish identically on $M_\delta\times \{\pi\}$;
  \item $\mc{D}^{\tilde{S}\otimes\tilde{E}}\tilde{u}-\frac{in}{2}\Psi_{\delta,\varepsilon}\tilde{u}=0$ on $\tilde{M}_\delta$;
  \item $\tilde{u}=-i\nu\cdot \tilde{u}$ on $\p_+\tilde{M}\times \{\pi\}$ and $\tilde{u}=i\nu\cdot \tilde{u}$ on $\p_-\tilde{M}\times \{\pi\}$;
  \item $(\n_{e_{n+1}}^{\tilde{S}\otimes \tilde{E}}\tilde{u})(\cdot,\pi)=0$, where $e_{n+1}=\frac{1}{r}\frac{\p}{\p t}$ is the unit vector on $M_\delta\times \{\pi\}$ tangent to $\mb{S}^1$.
\end{itemize}
For any vector field $X$ on $\tilde{M}_\delta$, denote
\begin{equation*}
  \tilde{P}_X\tilde{u}=\n^{\tilde{S}\otimes\tilde{E}}_{ X-\left\langle X,e_{n+1}\right\rangle e_{n+1}}\tilde{u}+\frac{i}{2}\Psi_{\delta,\varepsilon}(X-\left\langle X,e_{n+1}\right\rangle e_{n+1})\cdot \tilde{u}.
\end{equation*}
Using the same proof as \cite[Proposition 2.5 and (16)]{BBHW}, we obtain
\begin{align}\begin{split}\label{eqn1}
\begin{split}
  &\quad \int_{{M}_\delta\times\{\pi\}}|\tilde{P} \tilde{u}|^2+\frac{1}{4} \int_{{M}_\delta\times\{\pi\}} R|\tilde{u}|^2+\int_{{M}_\delta\times\{\pi\}}\left\langle\mathcal{R}^{\tilde{E}} \tilde{u}, \tilde{u}\right\rangle \\
& \quad\quad +\frac{n(n-1)}{4} \int_{{M}_\delta\times\{\pi\}} \Psi_{\delta, \varepsilon}^2|\tilde{u}|^2-\frac{i(n-1)}{2} \int_{{M}_\delta\times\{\pi\}}\left\langle\left(\nabla \Psi_{\delta, \varepsilon}\right) \cdot \tilde{u}, \tilde{u}\right\rangle \\
& =-\frac{1}{2} \int_{\partial_{+} {M}_\delta\times\{\pi\}}\left(H-(n-1) \Psi_{\delta, \varepsilon}\right)|\tilde{u}|^2-\frac{1}{2} \int_{\partial_{-} {M}_\delta\times\{\pi\}}\left(H+(n-1) \Psi_{\delta, \varepsilon}\right)|\tilde{u}|^2,
 \end{split}
\end{split}\end{align}
where $H$ denotes the mean curvature of $\p M_\delta$, defined as the sum of the principal curvatures. 

When restricted to $M_\delta\times\{\pi\}$, one obtains
\begin{equation*}
   \sum_{k=1}^{n}e_k\cdot\n^{\tilde{S}\otimes \tilde{E}}_{e_k}\tilde{u}-\frac{in}{2}\Psi_{\delta,\varepsilon}\tilde{u}=\left(\mc{D}^{\tilde{S}\otimes\tilde{E}}\tilde{u}-\frac{in}{2}\Psi_{\delta,\varepsilon}\tilde{u}\right)-e_{n+1}\cdot\n^{\tilde{S}\otimes \tilde{E}}_{e_{n+1}}\tilde{u}=0.
\end{equation*}
Next, we follow the method in \cite[Section 3.2]{hirsch-spectral-2023-arxiv}.  Denote 
\begin{equation*}
  v_i=e_i\cdot {\n}_{{e}_i}^{\tilde{S}\otimes \tilde{E}}\tilde{u}-\frac{i}{2}\Psi_{\delta,\varepsilon}\tilde{u}, \quad 1\leq i\leq n.
\end{equation*}
Then $\sum_{i=1}^n v_i=0$ on $M_\delta\times\{\pi\}$. Let $v=(v_1,\cdots,v_n)$, the Cauchy-Schwarz inequality gives 
 \begin{equation*}
  |\tilde{P}\tilde{u}|^2=\sum_{i=1}^n\left| {\n}_{{e}_i}^{\tilde{S}\otimes \tilde{E}}\tilde{u}+\frac{i}{2}\Psi_{\delta,\varepsilon}e_i\cdot\tilde{u}\right|^2=|v|^2\geq \frac{n}{n-1}|{\n}_{{e}_1}^{\tilde{S}\otimes \tilde{E}}\tilde{u}+\frac{i}{2}\Psi_{\delta,\varepsilon}e_1\cdot\tilde{u}|^2.
\end{equation*}
Denote $\beta=\frac{n}{n-1}-\frac{ c^{-1} }{4}>0$, it follows that
\begin{align}\begin{split}\label{eqn7}
\begin{split}
  |\tilde{P}\tilde{u}|^2 &\geq \frac{n}{n-1}| {\n}_{{e}_1}^{\tilde{S}\otimes \tilde{E}}\tilde{u}|^2+\frac{n}{(n-1)}\left\langle {\n}_{{e}_1}^{\tilde{S}\otimes \tilde{E}}\tilde{u},\frac{i}{2}\Psi_{\delta,\varepsilon}e_1\cdot\tilde{u}\right\rangle\\
  &\quad+\frac{n}{n-1}\left\langle \frac{i}{2}\Psi_{\delta,\varepsilon}e_1\cdot\tilde{u},{\n}_{{e}_1}^{\tilde{S}\otimes \tilde{E}}\tilde{u}\right\rangle+\frac{n}{4(n-1)}|\Psi_{\delta,\varepsilon}|^2|\tilde{u}|^2\\
  &=\frac{ c^{-1} }{4}| {\n}_{{e}_1}^{\tilde{S}\otimes \tilde{E}}\tilde{u}|^2+\beta\left|{\n}_{{e}_1}^{\tilde{S}\otimes \tilde{E}}\tilde{u}+\frac{n}{\beta(n-1)}\frac{i}{2}\Psi_{\delta,\varepsilon}e_1\cdot\tilde{u}\right|^2\\
  &\quad+\frac{n^2}{4}\underbrace{\left(\frac{1}{n(n-1)}-\frac{1}{\beta(n-1)^2}\right)}_{\beta_1}|\Psi_{\delta,\varepsilon}|^2|\tilde{u}|^2.\\
 \end{split}
\end{split}\end{align}
Let $e$ be a tangent vector defined by 
\begin{equation*}
  e= \begin{cases}
 	\frac{\n |\tilde{u}|}{|\n |\tilde{u}||}& \text{ for }\n|\tilde{u}|\neq 0, \\
 	0&\text{ for }\n|\tilde{u}|= 0.
 \end{cases}
\end{equation*}
Denote 
\begin{equation*}
  \mc{Q}_{e}\tilde{u}={\n}_{{e}}^{\tilde{S}\otimes \tilde{E}}\tilde{u}+\frac{n}{\beta(n-1)}\frac{i}{2}\Psi_{\delta,\varepsilon}e\cdot\tilde{u}.
\end{equation*}
From \eqref{eqn7}, the following inequality
\begin{equation}\label{eqn3}
 |\tilde{P}\tilde{u}|^2-\left( \frac{ c^{-1} }{4}|\n|\tilde{u}||^2+\frac{n^2\beta_1}{4}|\Psi_{\delta,\varepsilon}|^2|\tilde{u}|^2\right)\geq \beta\left|\mc{Q}_{e}\tilde{u}\right|^2
 \end{equation}
holds on $M_\delta\times\{\pi\}$. 
Using the divergence theorem, we obtain
\begin{align}\begin{split}\label{eqn4}
\begin{split}
 &\quad n\int_{M_\delta\times \{\pi\}}|\Psi_{\delta,\varepsilon}|^2|\tilde{u}|^2-i\int_{M_\delta\times \{\pi\}}\left\langle(\n \Psi_{\delta,\varepsilon})\cdot\tilde{u},\tilde{u}\right\rangle\\
 &=-i\int_{M_\delta\times\{\pi\}}\Psi_{\delta,\varepsilon}\left\langle \mc{D}^{\tilde{S}\otimes\tilde{E}}\tilde{u},\tilde{u}\right\rangle+i\int_{M_\delta\times\{\pi\}}\Psi_{\delta,\varepsilon}\left\langle \tilde{u},\mc{D}^{\tilde{S}\otimes\tilde{E}}\tilde{u}\right\rangle\\
 &\quad -i\int_{M_\delta\times \{\pi\}}\left\langle(\n \Psi_{\delta,\varepsilon})\cdot\tilde{u},\tilde{u}\right\rangle\\
  &=-i\int_{\p{M}_\delta\times \{\pi\}}\Psi_{\delta,\varepsilon}\left\langle\nu\cdot \tilde{u},\tilde{u}\right\rangle\\
  &=\int_{\p_+M_\delta\times \{\pi\}}\Psi_{\delta,\varepsilon}|\tilde{u}|^2-\int_{\p_-M_\delta\times \{\pi\}}\Psi_{\delta,\varepsilon}|\tilde{u}|^2.
  \end{split}
\end{split}\end{align}
Substituting \eqref{eqn4} into \eqref{eqn1}, one obtains
\begin{align}\begin{split}\label{eqn5}
\begin{split}
  & \quad \int_{{M}_\delta\times\{\pi\}}\beta\left|\mc{Q}_{e}\tilde{u}\right|^2+\int_{{M}_\delta\times\{\pi\}}\frac{1}{4}( c^{-1} |\n|\tilde{u}||^2+R|\tilde{u}|^2)\\
& \quad +\int_{{M}_\delta\times\{\pi\}}\left\langle\mathcal{R}^{\tilde{E}} \tilde{u}, \tilde{u}\right\rangle +\left(\frac{n(n-1)}{4} -\frac{n^2\beta_1}{4}\right)\int_{{M}_\delta\times\{\pi\}} \Psi_{\delta, \varepsilon}^2|\tilde{u}|^2\\
&\quad -\left(\frac{i(n-1)}{2} -\frac{in\beta_1}{2}\right)\int_{{M}_\delta\times\{\pi\}}\left\langle\left(\nabla \Psi_{\delta, \varepsilon}\right) \cdot \tilde{u}, \tilde{u}\right\rangle \\
& \leq -\frac{1}{2} \int_{\partial_{+} {M}_\delta\times\{\pi\}}\left(H-((n-1)-{n\beta_1}) \Psi_{\delta, \varepsilon}\right)|\tilde{u}|^2\\
&\quad -\frac{1}{2} \int_{\partial_{-} {M}_\delta\times\{\pi\}}\left(H+((n-1)-n\beta_1)\Psi_{\delta, \varepsilon}\right)|\tilde{u}|^2.
 \end{split}
\end{split}\end{align}
By the definition of spectral constant, we have
\begin{equation*}
  \int_{{M}_\delta\times\{\pi\}} ( c^{-1} |\nabla|\tilde{u}||^2+R|\tilde{u}|^2) \geq c^{-1}\Lambda_{c} \int_{M_\delta\times\{\pi\}}|\tilde{u}|^2. 
\end{equation*}
Denote $\beta_2=1-\frac{n\beta_1}{n-1}>0$. Then \eqref{eqn5} gives
\begin{align*}
&\quad \int_{{M}_\delta\times\{\pi\}}\beta\left|\mc{Q}_{e}\tilde{u}\right|^2+\int_{M_\delta\times\{\pi\}}\left\langle \mc{R}^{\tilde{E}}\tilde{u},\tilde{u}\right\rangle\\
	&\quad +\int_{M_\delta\times\{\pi\}}\left(\frac{c^{-1}{\Lambda}_ c }{4}+\frac{n(n-1)}{4}\beta_2\Psi_{\delta,\epsilon}^2-\frac{n-1}{2}\beta_2|\n\Psi_{\delta,\epsilon}|\right)|\tilde{u}|^2\\
	&\leq -\frac{1}{2} \int_{\partial_{+} {M}_\delta\times\{\pi\}}\left(H-(n-1)\beta_2\Psi_{\delta, \varepsilon}\right)|\tilde{u}|^2\\
	&\quad -\frac{1}{2} \int_{\partial_{-} {M}_\delta\times\{\pi\}}\left(H+(n-1)\beta_2\Psi_{\delta, \varepsilon}\right)|\tilde{u}|^2.
\end{align*}
Recall that $\Psi_{\delta,\varepsilon}=\psi_{\delta,\varepsilon}(\Theta)$, by the choice of $\psi_{\delta,\varepsilon}$, one has
\begin{equation*}
  H-(n-1)\beta_2\Psi_{\delta,\varepsilon}=H-(n-1)\beta_2\frac{\cos\varepsilon}{\sin\varepsilon}>0
\end{equation*}
on $\p_+\tilde{M}_\delta$ and 
\begin{equation*}
  H+(n-1)\beta_2\Psi_{\delta,\varepsilon}=H+(n-1)\beta_2\frac{\cos\varepsilon}{\sin\varepsilon}>0
\end{equation*}
on $\p_-\tilde{M}_\delta$. Hence 
\begin{align}\begin{split}\label{eqn8}
\begin{split}
&\quad \int_{{M}_\delta\times\{\pi\}}\beta\left|\mc{Q}_{e}\tilde{u}\right|^2+\int_{M_\delta\times\{\pi\}}\left\langle \mc{R}^{\tilde{E}}\tilde{u},\tilde{u}\right\rangle\\
	&\quad+\int_{M_\delta\times\{\pi\}}\left(\frac{c^{-1} \Lambda_c }{4}+\frac{n(n-1)}{4}\beta_2\Psi_{\delta,\epsilon}^2-\frac{n-1}{2}\beta_2|\n\Psi_{\delta,\epsilon}|\right)|\tilde{u}|^2\leq 0.
	\end{split}
\end{split}\end{align}

Denote by $\mu_1^2,\cdots,\mu_n^2,\mu_{n+1}^2$ ($\mu_{n+1}=0$) the eigenvalues of the metric $\tilde{f}^*g_{\mb{S}^n}$ with respect to the metric $\tilde{g}=g+r^2 g_{\mb{S}^1}$, $\mu_1\geq\mu_2\geq \cdots\geq \mu_n\geq \mu_{n+1}=0$. Using \eqref{eqn2}, one has
\begin{align}\begin{split}\label{eqn16}
\begin{split}
  \tilde{f}^*g_{\mb{S}^n}&=(\varphi\times \mathrm{id})^*h^*g_{\mb{S}^n}\\
  &=(\varphi\times \mathrm{id})^*(\cos^2(t)g_{\mathbb{S}^{n-1}}+(dt)^2)\\
  &=\cos^2(t)\varphi^*g_{\mathbb{S}^{n-1}}+(dt)^2\\
  &\leq \varphi^*g_{\mathbb{S}^{n-1}}+g_{\mb{S}^1}.
 \end{split}
\end{split}\end{align}
On the other hand, 
$$\tilde{g}=g+r^2 g_{\mb{S}^1}\geq a^2\sin^2(b\Theta)\varphi^* g_{\mathbb{S}^{n-1}}+r^2g_{\mb{S}^1}.$$
By taking $r\geq a\sin(b\Theta)$, we obtain 
\begin{equation}\label{eqn15}
  \mu_1,\cdots,\mu_{n-1}\leq \frac{1}{a\sin(b\Theta)},\quad \mu_n= \frac{1}{r}.
\end{equation}
 By using \cite[(4.6)]{llarull-sharp-1998}, see also \cite[Proposition A.1]{BBHW}, one has
 \begin{align}\begin{split}\label{eqn9}
\begin{split}
    \left\langle \mc{R}^{\tilde{E}}\tilde{u},\tilde{u}\right\rangle&\geq -\frac{1}{4}\sum_{1\leq k,\ell,k\neq\ell}\mu_k\mu_l|\tilde{u}|^2\\
    &\geq -\frac{(n-2)(n-1)}{4}\frac{1}{a^2\sin^2(b\Theta)}|\tilde{u}|^2-\frac{1}{r}\cdot\frac{n-1}{2a\sin(b\Theta)}|\tilde{u}|^2.
 \end{split}
\end{split}\end{align}
Moreover, since $|\n\Theta|\leq 1$, we obtain
\begin{equation}\label{eqn10}
 |\n\Psi_{\delta,\varepsilon}|\leq  \begin{cases}
 \frac{2a_1b}{n}\frac{1}{\sin^2(b(\Theta-\delta+\varepsilon))}
	&\text{for } \Theta\in [\delta,\frac{\pi}{3b}], \\
  \frac{2a_1b}{n}\frac{1}{\sin^2(b(\Theta+\delta-\varepsilon))}&\text{for }\Theta\in [\frac{2\pi}{3b},\frac{\pi}{b}-\delta],\\
 \frac{2a_1b}{n}\frac{1}{\sin^2(b\Theta)}+K\delta &\text{for }\Theta\in [\frac{\pi}{3b},\frac{2\pi}{3b}].
 \end{cases}
\end{equation}
By using \eqref{eqn9} and \eqref{eqn10}, for $\Theta\in [\delta,\frac{\pi}{3b}]$, one has
\begin{align*}
\begin{split}
& \quad \left(\frac{c^{-1}\Lambda}{4}+\frac{n(n-1)}{4}\beta_2\Psi_{\delta,\epsilon}^2-\frac{n-1}{2}\beta_2|\n\Psi_{\delta,\epsilon}|\right)|\tilde{u}|^2+\left\langle \mc{R}^{\tilde{E}}\tilde{u},\tilde{u}\right\rangle\\
&\geq \left(\frac{c^{-1}\Lambda}{4}+\frac{n(n-1)}{4}\beta_2\frac{4a_1^2}{n^2}\frac{\cos^2(b(\Theta-\delta+\varepsilon))}{\sin^2(b(\Theta-\delta+\varepsilon))}\right.\\
&\quad \left.-\frac{n-1}{2}\beta_2\frac{2a_1b}{n}\frac{1}{\sin^2(b(\Theta-\delta+\varepsilon))}\right)|\tilde{u}|^2\\
&\quad  -\frac{(n-2)(n-1)}{4}\frac{1}{a^2\sin^2(b\Theta)}|\tilde{u}|^2-\frac{n-1}{2ra\sin(b\Theta)}|\tilde{u}|^2\\
&\geq-\frac{n-1}{2ra\sin (b\Theta)}|\tilde{u}|^2+ \left[\left(\frac{c^{-1}\Lambda}{4}-\frac{(n-1)\beta_2a_1^2}{n}\right)\right.\\
&\left.+\frac{n-1}{n}\left(\beta_2a_1^2-\beta_2a_1b-\frac{n(n-2)}{4a^2}\right)\frac{1}{\sin^2(b(\Theta-\delta+\varepsilon))}\right]|\tilde{u}|^2\\
&\quad \\
&= -\frac{n-1}{2ra\sin (b\Theta)}|\tilde{u}|^2,
 \end{split}
\end{align*}
where the last equality holds by the choice of $a,b$ such that
\begin{equation*}
  \frac{c^{-1}\Lambda}{4}-\frac{(n-1)\beta_2a_1^2}{n}=0=\beta_2a_1^2-\beta_2a_1b-\frac{n(n-2)}{4a^2}.
\end{equation*}
Similarly,
 \begin{align*}
\begin{split}
   &\quad \left(\frac{c^{-1}\Lambda}{4}+\frac{n(n-1)}{4}\beta_2\Psi_{\delta,\epsilon}^2-\frac{n-1}{2}\beta_2|\n\Psi_{\delta,\epsilon}|\right)|\tilde{u}|^2+\left\langle \mc{R}^{\tilde{E}}\tilde{u},\tilde{u}\right\rangle\\
   &\geq -\frac{n-1}{2ra\sin (b\Theta)}|\tilde{u}|^2
 \end{split}
\end{align*}
for $\Theta\in [\frac{2\pi}{3b},\frac{\pi}{b}-\delta]$, and 
\begin{align*}
\begin{split}
 &  \left(\frac{c^{-1}\Lambda}{4}+\frac{n(n-1)}{4}\beta_2\Psi_{\delta,\epsilon}^2-\frac{n-1}{2}\beta_2|\n\Psi_{\delta,\epsilon}|\right)|\tilde{u}|^2+\left\langle \mc{R}^{\tilde{E}}\tilde{u},\tilde{u}\right\rangle\\
 &\geq -\frac{n-1}{2r}\frac{1}{a\sin (b\Theta)}|\tilde{u}|^2-L\delta|\tilde{u}|^2
 \end{split}
\end{align*}
for $\Theta\in[\frac{\pi}{3b},\frac{2\pi}{3b}]$, where $L>0$ is a constant independent of $\delta,\varepsilon$. Therefore, 
\begin{align}\begin{split}\label{eqn14}
\begin{split}
 & \quad \left(\frac{c^{-1}\Lambda}{4}+\frac{n(n-1)}{4}\beta_2\Psi_{\delta,\epsilon}^2-\frac{n-1}{2}\beta_2|\n\Psi_{\delta,\epsilon}|\right)|\tilde{u}|^2+\left\langle \mc{R}^{\tilde{E}}\tilde{u},\tilde{u}\right\rangle\\
 &\geq -\frac{n-1}{2r}\frac{1}{a\sin (b\Theta)}|\tilde{u}|^2-L\delta\cdot 1_{\{\Theta\in [\frac{\pi}{3b},\frac{2\pi}{3b}]\}}|\tilde{u}|^2
 \end{split}
\end{split}\end{align}
on $M_\delta\times\{\pi\}$. Following the argument in \cite{WX}, we obtain the following proposition. 
\begin{proposition}\label{prop1}
	For any given $x_0\in M_\delta$ and small $\delta_0>0$, let $B_{\delta_0}(x_0)\subset M_\delta$ be the open ball of radius $\delta_0$ centered at $x_0$. There exists no positive constant $C$ such that the following inequality
	 \begin{align}\begin{split}\label{eqn11}
\begin{split}
 & \quad \left(\frac{c^{-1}{\Lambda}_{ c }}{4}+\frac{n(n-1)}{4}\beta_2\Psi_{\delta,\epsilon}^2-\frac{n-1}{2}\beta_2|\n\Psi_{\delta,\epsilon}|\right)|\tilde{u}|^2+\left\langle \mc{R}^{\tilde{E}}\tilde{u},\tilde{u}\right\rangle\\
 &\geq \frac{1}{C}1_{\{B_{\delta_0}(x_0)\times \{\pi\}\}}|\tilde{u}|^2 -\frac{n-1}{2r}\frac{1}{a\sin (b\Theta)}|\tilde{u}|^2-L\delta\cdot 1_{\{\Theta\in [\frac{\pi}{3b},\frac{2\pi}{3b}]\}}|\tilde{u}|^2
 \end{split}
\end{split}\end{align}
holds on $M_\delta\times \{\pi\}$ for any small $\delta,\varepsilon$ and any large $r$.
\end{proposition}
\begin{proof}
	 Suppose that there exists some constant $C>0$ such that \eqref{eqn11} holds for any small $\delta,\varepsilon$ and any large $r$, by \eqref{eqn8}, one has
	 \begin{align*}
  &\quad \int_{{M}_\delta\times\{\pi\}}\left(\beta\left|\mc{Q}_{e}\tilde{u}\right|^2+\frac{1}{C}1_{\{B_{\delta_0}(x_0)\times \{\pi\}\}}|\tilde{u}|^2\right)\\
  &\leq  \int_{{M}_\delta\times\{\pi\}}\left(\frac{n-1}{2r}\frac{1}{a\sin (b\Theta)}|\tilde{u}|^2+L\delta\cdot 1_{\{\Theta\in [\frac{\pi}{3b},\frac{2\pi}{3b}]\}}|\tilde{u}|^2 \right).
\end{align*}
On the other hand, by \cite[Proposition A.1]{BC}, one has
\begin{align}\begin{split}\label{eqn12}
\begin{split}
  \int_{M_\delta\times\{\pi\}}|\tilde{u}|^2&\leq C_1\int_{M_\delta\times\{\pi\}}(|\n|\tilde{u}||-K_1|\tilde{u}|)^2_{+}+C_1\int_{M_\delta\times \{\pi\}}1_{\{B_{\delta_0}(x_0)\times \{\pi\}\}}|\tilde{u}|^2\\
  &\leq C_1\int_{M_\delta\times\{\pi\}}|\mc{Q}_{e}\tilde{u}|^2+C_1\int_{M_\delta\times \{\pi\}}1_{\{B_{\delta_0}(x_0)\times \{\pi\}\}}|\tilde{u}|^2,
 \end{split}
\end{split}\end{align}
where $K_1=\max_{M_\delta}\frac{n}{2\beta (n-1)}|\Psi_{\delta,\varepsilon}|$, and so $C_1=C_1(\delta,\varepsilon)$ depends on $\delta,\varepsilon$. Similarly, 
\begin{align*}
\begin{split}
  \int_{M_\delta\times\{\pi\}}1_{\{\Theta\in [\frac{\pi}{3b},\frac{2\pi}{3b}]\}}|\tilde{u}|^2
  &\leq C_2\int_{M_\delta\times\{\pi\}}|\mc{Q}_{e}\tilde{u}|^2+C_2\int_{M_\delta\times \{\pi\}}1_{\{B_{\delta_0}(x_0)\times \{\pi\}\}}|\tilde{u}|^2,
 \end{split}
\end{align*}
where $C_2$ is independent of $\delta,\varepsilon$ since this term $\frac{n}{2\beta (n-1)}|\Psi_{\delta,\varepsilon}|$ is uniformly bounded on a compact, connected Lipschitz domain contains the subsets $\{\Theta\in [\frac{\pi}{3b},\frac{2\pi}{3b}]\}$ and $B_{\delta_0}(x_0)$. 
Hence 
\begin{align*}
\begin{split}
 &\quad \min\{\beta,C^{-1}\}\int_{{M}_\delta\times\{\pi\}}\left(\left|\mc{Q}_{e}\tilde{u}\right|^2+1_{\{B_{\delta_0}(x_0)\times \{\pi\}\}}|\tilde{u}|^2\right)\\
 &\leq \int_{{M}_\delta\times\{\pi\}}\left(\beta\left|\mc{Q}_{e}\tilde{u}\right|^2+\frac{1}{C}1_{\{B_{\delta_0}(x_0)\times \{\pi\}\}}|\tilde{u}|^2\right)\\
 &\leq  \int_{{M}_\delta\times\{\pi\}}\left(\frac{n-1}{2r}\frac{1}{a\sin (b\Theta)}|\tilde{u}|^2+L\delta\cdot 1_{\{\Theta\in [\frac{\pi}{3b},\frac{2\pi}{3b}]\}}|\tilde{u}|^2 \right)\\
 &\leq  \frac{C_3}{r}\int_{{M}_\delta\times\{\pi\}}|\tilde{u}|^2+L\delta C_2\left(\int_{M_\delta\times\{\pi\}}|\mc{Q}_{e}\tilde{u}|^2+\int_{M_\delta\times \{\pi\}}1_{\{B_{\delta_0}(x_0)\times \{\pi\}\}}|\tilde{u}|^2\right)\\
 &\leq (\frac{C_3C_1}{r}+L\delta C_2)\left(\int_{M_\delta\times\{\pi\}}|\mc{Q}_{e}\tilde{u}|^2+\int_{M_\delta\times \{\pi\}}1_{\{B_{\delta_0}(x_0)\times \{\pi\}\}}|\tilde{u}|^2\right).
 \end{split}
\end{align*}
For any given $\delta,\varepsilon$ such that $L\delta C_2\leq \frac{1}{3}\min\{\beta,C^{-1}\}$, and taking $r$ large enough such that $C_3C_1r^{-1}\leq \frac{1}{3}\min\{\beta,C^{-1}\}$, we have 
\begin{equation*}
  \int_{M_\delta\times\{\pi\}}|\mc{Q}_{e}\tilde{u}|^2=0=\int_{M_\delta\times \{\pi\}}1_{\{B_{\delta_0}(x_0)\times \{\pi\}\}}|\tilde{u}|^2.
\end{equation*}
From \eqref{eqn12}, we conclude that $\tilde{u}\equiv 0$ on $M_\delta\times\{\pi\}$, which is a contradiction since the zero set of a nontrivial harmonic spinor has codimension $\geq 2$ due to B\"ar \cite{Bar}.
\end{proof}
From  \eqref{eqn14}, one has
 \begin{align}\begin{split}
\begin{split}
 & \quad \left(\frac{c^{-1}{\Lambda}_ c }{4}+\frac{n(n-1)}{4}\beta_2\Psi_{\delta,\epsilon}^2-\frac{n-1}{2}\beta_2|\n\Psi_{\delta,\epsilon}|\right)|\tilde{u}|^2+\left\langle \mc{R}^{\tilde{E}}\tilde{u},\tilde{u}\right\rangle\\
 &\geq \frac{1}{4}(c^{-1}{\Lambda}_ c -c^{-1}\Lambda)|\tilde{u}|^2-\frac{n-1}{2r}\frac{1}{a\sin (b\Theta)}|\tilde{u}|^2-L\delta\cdot 1_{\{\Theta\in [\frac{\pi}{3b},\frac{2\pi}{3b}]\}}|\tilde{u}|^2.
 \end{split}
\end{split}\end{align}
By Proposition \ref{prop1}, we obtain 
\begin{equation*}
{\Lambda}_ c =\Lambda.
\end{equation*}

From \eqref{eqn13}, we know $|\n\Theta|\leq 1$ and the inequality is strict unless $d\varphi(\n\Theta)=0$. If $|\n\Theta|\not\equiv 0$, then $1-|\n\Theta|\geq \frac{1}{C}1_{\{B_{\delta_0}(x_0)\}}$ for some $x_0\in M_\delta$, $\delta_0>0$ and $C>0$. Hence 
\begin{align*}
\begin{split}
  |\n\Psi_{\delta,\varepsilon}|&=|\psi'_{\delta,\varepsilon}||\n\Theta|=-|\psi'_{\delta,\varepsilon}|(1-|\n\Theta|)+|\psi'_{\delta,\varepsilon}|\\
  &\leq -|\psi'_{\delta,\varepsilon}|\frac{1}{C}1_{\{B_{\delta_0}(x_0)\}}+|\psi'_{\delta,\varepsilon}|.
 \end{split}
\end{align*}
 By the definition of $\psi_{\delta,\varepsilon}$, one has $|\psi'_{\delta,\varepsilon}|\geq \frac{2a_1b}{n}-K\delta$. By taking sufficiently small $\delta>0$ satisfying $\delta\leq \frac{a_1b}{Kn}$, one has
  \begin{equation*}
    |\n\Psi_{\delta,\varepsilon}|\leq -\frac{a_1b}{n}\frac{1}{C}1_{\{B_{\delta_0}(x_0)\}}+|\psi'_{\delta,\varepsilon}|.
\end{equation*}
Using the same proof as \eqref{eqn14}, we obtain 
\begin{align*}
\begin{split}
 & \quad \left(\frac{c^{-1}\Lambda}{4}+\frac{n(n-1)}{4}\beta_2\Psi_{\delta,\epsilon}^2-\frac{n-1}{2}\beta_2|\n\Psi_{\delta,\epsilon}|\right)|\tilde{u}|^2+\left\langle \mc{R}^{\tilde{E}}\tilde{u},\tilde{u}\right\rangle\\
 &\geq \frac{a_1 b(n-1)\beta_2}{2nC}1_{\{B_{\delta_0}(x_0)\times \{\pi\}\}}|\tilde{u}|^2 -\frac{n-1}{2r}\frac{1}{a\sin (b\Theta)}|\tilde{u}|^2-L\delta\cdot 1_{\{\Theta\in [\frac{\pi}{3b},\frac{2\pi}{3b}]\}}|\tilde{u}|^2,
 \end{split}
\end{align*}
 which is a contradiction to Proposition \ref{prop1}. Hence $|\n\Theta|\equiv 1$, and so $d\varphi(\n\Theta)\equiv 0$. Thus 
 \begin{equation}\label{eqn20}
  d\tilde{f}(\n\Theta)\equiv 0.
\end{equation}
Recall that $\mu_{n-1}\leq \cdots\leq \mu_1\leq \frac{1}{a\sin(b\Theta)}$, see \eqref{eqn15}. If $\mu_{n-1}\not\equiv \frac{1}{a\sin(b\Theta)}$, then 
\begin{equation*}
  \frac{1}{a\sin(b\Theta)}-\mu_{n-1}\geq \frac{1}{C}1_{\{B_{\delta_0}(x_0)\}} 
\end{equation*}
for some $x_0\in M_\delta$, $\delta_0>0$ and $C>0$. Thus 
 \begin{align}
\begin{split}
    \left\langle \mc{R}^{\tilde{E}}\tilde{u},\tilde{u}\right\rangle&\geq -\frac{1}{4}\sum_{1\leq k,\ell,k\neq\ell}\mu_k\mu_l|\tilde{u}|^2\\
    &\geq -\frac{(n-2)(n-1)}{4}\frac{1}{a^2\sin^2(b\Theta)}|\tilde{u}|^2-\frac{1}{r}\cdot\frac{n-1}{2a\sin(b\Theta)}|\tilde{u}|^2\\
    &\quad +\frac{n-1}{2Ca}1_{\{B_{\delta_0}(x_0)\}}.
\end{split}\end{align}
Using the same proof as \eqref{eqn14}, we obtain the inequality \eqref{eqn11}, which contradicts  Proposition \ref{prop1}. Therefore,\begin{equation}\label{eqn19}
  \mu_{n-1}= \cdots=\mu_1=\frac{1}{a\sin(b\Theta)},\quad\mu_n=\frac{1}{r}
\end{equation}
on $M_\delta\times\{\pi\}$ by  \eqref{eqn15}. Denote 
\begin{equation*}
  f: M \to \mb{S}^n,\quad f(x)=\tilde{f}(x,\pi).
\end{equation*}
By \eqref{eqn20} and \eqref{eqn19}, one has
\begin{equation}\label{eqn17}
f^*g_{\mb{S}^n}=(\mathrm{id}\times \{\pi\})^*  \tilde{f}^*g_{\mb{S}^n}=\frac{1}{a^2\sin^2(b\Theta)}(g-d\Theta\otimes d\Theta).
\end{equation}
On the other hand, by \eqref{eqn16}, one has
\begin{equation}\label{eqn18}
  f^*g_{\mb{S}^n}=(\mathrm{id}\times \{\pi\})^*  \tilde{f}^*g_{\mb{S}^n}=\varphi^*g_{\mathbb{S}^{n-1}}.
\end{equation}
Combining \eqref{eqn17} and \eqref{eqn18}, we conclude that
\begin{equation*}
  g=  d\Theta\otimes d\Theta+a^2\sin^2(b\Theta)\varphi^*g_{\mathbb{S}^{n-1}}=\Phi^*g_0.
\end{equation*}
This means that $\Phi$ is a local isometry, and since the target is simply connected, it follows that $\Phi$ is a global Riemannian isometry. The proof of the main Theorem \ref{thm1} for even $n$ is complete.

When $n$ is odd, the proof is somewhat simpler, as we can work with $M$ and $M_\delta$ directly and we do not need to consider the Cartesian product with $\mb{S}^1$. Let $S$ denote the spinor bundle over $M$, and let $E_0$ denote the spinor bundle over the round sphere $\mb{S}^{n-1}$. Since $n-1$ is even, we may decompose $E_0=E_0^{+} \oplus E_0^{-}$, where $E_0^{+}$and $E_0^{-}$denote the eigenbundles of the volume form. 

For each $\delta\in \Delta$, let $\nu$ be the outward unit normal vector and consider the indices of the following operators:
\begin{itemize}
  \item  Let $\mathrm{ind}_1$ denote the index of the Dirac operator on $S \otimes \varphi^* E_0^{+}$ with boundary conditions $u=-i \nu \cdot u$ on $\partial_{+} M_\delta$ and $u=i \nu \cdot u$ on $\partial_{-} M_\delta$;
\item  Let $\mathrm{ind}_2$ denote the index of the Dirac operator on $S \otimes \varphi^* E_0^{+}$ with boundary conditions $u=i \nu \cdot u$ on $\partial_{+} M_\delta$ and $u=-i \nu \cdot u$ on $\partial_{-} M_\delta$;
\item Let $\mathrm{ind}_3$ denote the index of the Dirac operator on $S \otimes \varphi^* E_0^{-}$ with boundary conditions $u=-i \nu \cdot u$ on $\partial_{+} M_\delta$ and $u=i \nu \cdot u$ on $\partial_{-} M_\delta$;
\item Let $\mathrm{ind}_4$ denote the index of the Dirac operator on $S \otimes \varphi^* E_0^{-}$ with boundary conditions $u=i \nu \cdot u$ on $\partial_{+} M_\delta$ and $u=-i \nu \cdot u$ on $\partial_{-}M_\delta$.
\end{itemize}
 Denote
\begin{equation*}
  \Delta_i=\{\delta\in\Delta:\mathrm{ind}_i>0\},\quad 1\leq i\leq 4.
\end{equation*}
Using the same arguments as above, we may assume that $\Delta_1$ contains $0$ in its closure.

Denote $E=\varphi^*E_0^+$. Let $\n^{S\otimes E}$ denote the natural connection induced from $S$ and $E$, and let $\mc{D}^{S\otimes E}$ denote the Dirac operator acting on the sections of $S\otimes E$. For any  $\delta\in\Delta_1$ and using the deformation invariance of the index, we can find a section $u\in C^\infty(M_\delta,S\otimes E)$ such that
\begin{itemize}
  \item $u$ does not vanish identically;
  \item $\mc{D}^{S\otimes E}u-\frac{in}{2}\Psi_{\delta,\varepsilon}u=0$ on $M_\delta$;
  \item $u=-i\nu\cdot u$ on $\p_+M_\delta$ and $u=i\nu\cdot u$ on $\p_-M_\delta$.
\end{itemize}
Using the same proof as \eqref{eqn8} and \eqref{eqn14}, we obtain 
\begin{align}\begin{split}
\begin{split}
&\quad \int_{{M}_\delta}\beta\left|\mc{Q}_{e}u\right|^2+\int_{M_\delta}\left\langle \mc{R}^{E}u,u\right\rangle\\
	&\quad+\int_{M_\delta}\left(\frac{c^{-1} \Lambda_c }{4}+\frac{n(n-1)}{4}\beta_2\Psi_{\delta,\epsilon}^2-\frac{n-1}{2}\beta_2|\n\Psi_{\delta,\epsilon}|\right)|u|^2\leq 0.
	\end{split}
\end{split}\end{align}
and 
\begin{align}\begin{split}\label{eqn21}
\begin{split}
 & \quad \left(\frac{c^{-1}\Lambda}{4}+\frac{n(n-1)}{4}\beta_2\Psi_{\delta,\epsilon}^2-\frac{n-1}{2}\beta_2|\n\Psi_{\delta,\epsilon}|\right)|u|^2+\left\langle \mc{R}^{E}u,u\right\rangle\\
 &\geq -L\delta\cdot 1_{\{\Theta\in [\frac{\pi}{3b},\frac{2\pi}{3b}]\}}|u|^2
 \end{split}
\end{split}\end{align}
on $M_\delta$. Similar to Proposition \ref{prop1}, we have the following key proposition. 
\begin{proposition}
	For any given $x_0\in M_\delta$ and small $\delta_0>0$, let $B_{\delta_0}(x_0)\subset M_\delta$ be the open ball of radius $\delta_0$ centered at $x_0$. There exists no positive constant $C$ such that the following inequality
	 \begin{align}\begin{split}\label{eqn22}
\begin{split}
 & \quad \left(\frac{c^{-1}{\Lambda}_{ c }}{4}+\frac{n(n-1)}{4}\beta_2\Psi_{\delta,\epsilon}^2-\frac{n-1}{2}\beta_2|\n\Psi_{\delta,\epsilon}|\right)|u|^2+\left\langle \mc{R}^{E}u,u\right\rangle\\
 &\geq \frac{1}{C}1_{\{B_{\delta_0}(x_0)\}}|u|^2 -L\delta\cdot 1_{\{\Theta\in [\frac{\pi}{3b},\frac{2\pi}{3b}]\}}|u|^2
 \end{split}
\end{split}\end{align}
holds on $M_\delta$ for any small $\delta,\varepsilon$.
\end{proposition}
By \eqref{eqn21} and \eqref{eqn22}, and using the same arguments as the even-dimensional case, we can conclude that 
$d\varphi(\n\Theta)\equiv 0$ and 
\begin{equation}
  \mu_{n-1}= \cdots=\mu_1=\frac{1}{a\sin(b\Theta)}
\end{equation}
on $M_\delta$, where $\mu_i,1\leq i\leq n-1$, are the eigenvalues of $\varphi^*g_{\mb{S}^{n-1}}$ with respect to $g$. Hence 
\begin{equation*}
  g=  d\Theta\otimes d\Theta+a^2\sin^2(b\Theta)\varphi^*g_{\mathbb{S}^{n-1}}=\Phi^*g_0.
\end{equation*}
This means that $\Phi$ is a local isometry, and since the target is simply connected, it follows that $\Phi$ is a global Riemannian isometry. The proof of the main Theorem \ref{thm1} for odd $n$ is complete. 

Now we can also give a quick proof of Theorem \ref{thm1-mu}.
\begin{proof}[Proof of Theorem \ref{thm1-mu}]
  The proof is very similar to that of Theorem \ref{thm1}. We mention the modifications.
  To construct the auxiliary function $\psi_{\delta,\epsilon}$, we perturb the function $f$ given in \eqref{f in mu}. Another ingredient is to replace $\Lambda_c \geq \Lambda$ with the new condition $\Lambda_{c,\mu\circ \Phi}\geq 0$.
  \end{proof}

\section{spacetime harmonic function proof}\label{sec:harmonic proof}

In this section we prove Theorem \ref{thm1} in dimension 3 using the spacetime harmonic function technique. We also show an integral inequality in Proposition \ref{final integral estimate} without assuming the lower bound $\Lambda_c \geq \Lambda$.
We leave the statement and proof of Theorem \ref{thm1-mu} to the interested readers.
We only give a description of the model metrics and related functions at the end of the section.

First, we write down again explicitly the following constants $c$, $a$ and $b$ given earlier and an extra auxiliary constant $\alpha$ in dimension 3:
\begin{equation}
  c > \tfrac{1}{4}, \alpha = \sqrt{\frac{\Lambda (6 - c^{- 1})}{2 c (8 - c^{-
  1})}}, a = \sqrt{\frac{6 c - 1}{\Lambda}}, b = \tfrac{\sqrt{\Lambda} (4 -
  c^{- 1})}{\sqrt{2 c (6 - c^{- 1}) (8 - c^{- 1})}} . \label{constants}
\end{equation}
Let $\xi$ be the $C^1$ function be given by $\xi= \tfrac{2(6-c^{-1})}{3(c^{-1}-4)}\phi^{-1}\phi'$, it is
readily checked that
$\xi$ satisfies the ordinary
differential equation
\begin{equation}
  1 + \tfrac{9}{4 \alpha^2} \xi^2 - \tfrac{3}{2 \alpha^2} \xi' = 2 c
  \Lambda^{- 1} \phi^{- 2} \label{eq:ode}
\end{equation}
with $\xi' \geq 0$ on $(0, \tfrac{\pi}{b})$. We can view $\xi$ as a function on $M$, we set $f=\xi(s(x))$.

\begin{theorem}\label{harmonic main}
Assume that on $M$, the function $f$ is given as above, in addition, we assume that there exist $c_1>0$ and $c_2>0$ such that 
\begin{align}
  \Delta s- \frac{2(6-c^{-1})}{4-c^{-1}}\sqrt{g^{ss}}\phi^{-1}\phi' &= -c_1 s^{-1}+O(1)\text{ near } s =0, \label{technical0}\\
  \Delta s- \frac{2(6-c^{-1})}{4-c^{-1}}\sqrt{g^{ss}}\phi^{-1}\phi' &= c_2 (\tfrac{\pi}{b}-s)^{-1}+O(1)\text{ near } s=\tfrac{\pi}{b}. \label{technical1}
\end{align}
Let $g$ satisfy
  \begin{enumerate}
    \item $g \geq g_0$,
    \item $\inf_M R_g > - \infty$, $c > \tfrac{1}{4}$ and $\Lambda_c
    \geq \Lambda > 0$.
  \end{enumerate}
  Then $g = g_0 $.
\end{theorem}

\begin{remark}
This theorem is a special case when $\Phi:(M,g)\to (M,g_0)$ is the identity. It is also possible to obtain for general $\Phi$. We encourage the readers to consult also \cite[Theorem 2.7]{hirsch-rigid-2022-arxiv}.
\end{remark}

\begin{remark}
  For $g_0$, the constants $c_1$ and $c_2$ in \eqref{technical0} and  \eqref{technical1} are $\tfrac{4}{4-c^{-1}}$,
  so the assumptions \eqref{technical0} and  \eqref{technical1} are reasonable.
\end{remark}

\subsection{Integral formula}
We recall an integral formula involving the spacetime harmonic function.

\begin{proposition}[{\cite[Proposition 3.2]{hirsch-spacetime-2022}}]
  \label{lm:hkk}Let $u$ be the solution to
\begin{align}\begin{split}
\Delta u + 3 f | \nabla u| & = 0 \text{ in } M, \\
u & = c_{\pm} \text{ on } \partial M,
\end{split}\end{align}
  where $f$ is Lipschitz, then $u \in C^{2, \alpha} \cap W^{3, p}$ and
\begin{align}\begin{split}
& - \int_{\partial_{\pm} M} 2 | \nabla u| (2 f \pm H) \\
\geq & \int_M (\tfrac{| \bar{\nabla}^2 u|^2}{| \nabla u|} + (R_g + 6
f^2) - 4 \langle \nabla f, \nabla u \rangle) - \int_{c_-}^{c_+} 4 \pi \chi
(\Sigma_t) \mathrm{d} t, \label{eq:hkk integral formula}
\end{split}\end{align}
  where $\bar{\nabla}^2 u = \nabla^2 u + f g_{i j} | \nabla u|$ and $\Sigma_t$ is the level set $\{u = t\}$.
\end{proposition}

We give the following integral inequality which might be of independent interest.

\begin{theorem} \label{final integral estimate}
Let $(M^3,g)$ and $f$ be given as in Theorem \ref{harmonic main}.
  Then there exists a spacetime harmonic function $u$ satisfying
  \[ \Delta u + 3 f | \nabla u| = 0 \]
  with $u (p_{\pm}) = \pm 1$, $| \nabla u|^{\tfrac{1}{2}} \in H^1_0 (M)$ and the
  following integral inequality
\begin{align}\begin{split}
& 4 \pi \int_{- 1}^1 \chi (\Sigma_t) \mathrm{d} t + \int_M c^{- 1}
(\Lambda - \Lambda_c) | \nabla u| \\
\geq & \int_M 2 \phi^{- 2} | \nabla u| + \int_M (6 - c^{- 1})  \left|
\nabla | \nabla u|^{\frac{1}{2}} + \frac{3}{6 - c^{- 1}} f| \nabla u|^{-
\frac{1}{2}} \nabla u \right|^2 . \label{final crucial integral inequality}
\end{split}\end{align}
\end{theorem}
\begin{remark}
Actually, we have $\lim_{x \to p_{\pm}} | \nabla u| = 0$. 
\end{remark}

Then Theorem \ref{harmonic main} can be obtained from the
inequality \eqref{final crucial integral inequality} by applying the condition
$\Lambda_c \geq \Lambda$.

\subsection{Construction of perturbed functions $f_{\varepsilon, i}$}

Let $0 < 2 \varepsilon < \tfrac{\pi}{b}$ and define
\[ M_{\varepsilon} = \{x \in M : \varepsilon \leq s (x) \leq
   \tfrac{\pi}{b} - \varepsilon\} . \]
Then $\{M_{\varepsilon} \}_{0 < \varepsilon < \tfrac{\pi}{2 b}}$ is an
exhaustion of $M$. We fix constants $2 \varepsilon < r_0 < r_1 < \tfrac{\pi}{2
b}$, for each sufficiently small $\varepsilon$ and sufficiently large $i$, we
construct the following piecewise linear function
\begin{align}\begin{split}
h_{\varepsilon, i} (s) = \left\{\begin{array}{ll}
i^{- 1} & \text{ for } \varepsilon \leq s \leq 2 \varepsilon ;\\
s + i^{- 1} - 2 \varepsilon & \text{ for } 2 \varepsilon \leq s
\leq r_0 ;\\
\tfrac{r_1 - (r_0 + i^{- 1} - 2 \varepsilon)}{r_1 - r_0} (s - r_1) + r_1 &
\text{ for } r_0 \leq s \leq r_1 ;\\
s & \text{ for } r_1 \leq s \leq \tfrac{\pi}{b} - r_1 ;\\
\tfrac{r_1 - (r_0 + i^{- 1} - 2 \varepsilon)}{r_1 - r_0} (s -
\tfrac{\pi}{b} + r_1) + \tfrac{\pi}{b} - r_1 & \text{ for } \tfrac{\pi}{b}
- r_1 \leq s \leq \tfrac{\pi}{b} - r_0 ;\\
s - (\tfrac{\pi}{b} - 2 \varepsilon) + (\tfrac{\pi}{b} - i^{- 1}) & \text{
for } \tfrac{\pi}{b} - r_0 \leq s \leq \tfrac{\pi}{b} - 2
\varepsilon ;\\
\tfrac{\pi}{b} - i^{- 1} & \text{ for } \tfrac{\pi}{b} - 2 \varepsilon
\leq s \leq \tfrac{\pi}{b} - \varepsilon .
\end{array}\right. &
\end{split}\end{align}
on $[\varepsilon, \tfrac{\pi}{b} - \varepsilon]$.

We can also regarded $\xi$ as a function on $M$ in the sense that $\xi$ takes
the constant value $\xi (s)$ in each $s$-level set.

We define now
\begin{equation}
  f_{\varepsilon, i} = \xi (h_{\varepsilon, i} (s (x))), x \in
  \bar{M}_{\varepsilon} . \label{f e j}
\end{equation}
Obviously, $f_{\varepsilon, i}$ are Lipschitz. We have different lower bounds
for the quantity
\[ 1 + \tfrac{9}{4 \alpha^2} f_{\varepsilon, i}^2 - \tfrac{3}{2 \alpha^2} |
   \nabla f_{\varepsilon, i} | \]
for different values of $s (x)$.

The easiest case is when $r_1 \leq s (x) \leq \tfrac{\pi}{b} - r_1$.
In this case, $f_{\varepsilon, i} (x) = \xi (s (x))$, using that $g \geq
\bar{g}$, we have
\begin{equation}
  - | \nabla f_{\varepsilon, i} | = - \xi' | \nabla s | \geq - \xi'  |
  \bar{\nabla} s |_{\bar{g}} = - \xi' . \label{gradient f e i}
\end{equation}
Since $\xi$ satisfies \eqref{eq:ode}, we see that
\begin{equation}
  1 + \tfrac{9}{4 \alpha^2} f_{\varepsilon, i}^2 - \tfrac{3}{2 \alpha^2} |
  \nabla f_{\varepsilon, i} | \geq 2 c \Lambda^{- 1} \phi^{- 2} \text{
  for } s (x) \in (r_1, \tfrac{\pi}{b} - r_1) . \label{eq:outside r1
  neighborhood}
\end{equation}
For other values of $s (x)$, the procedures are similar. We do it in the
following. For $s (x) \in (\varepsilon, 2 \varepsilon)$, we have
$f_{\varepsilon, i} = \xi (\tfrac{1}{i})$, $| \nabla f_{\varepsilon, i} | = 0$
and
\begin{align}\begin{split}
& 1 + \tfrac{9}{4 \alpha^2} f_{\varepsilon, i}^2 - \tfrac{3}{2 \alpha^2} |
\nabla f_{\varepsilon, i} | \\
\geq & 1 + \tfrac{9}{4 \alpha^2} (\xi (\tfrac{1}{i}))^2 - \xi'
(\tfrac{1}{i}) \\
\geq & 2 c \Lambda^{- 1} \tfrac{1}{\phi (\tfrac{1}{i})^2} .
\end{split}\end{align}
For $s (x) \in (2 \varepsilon, r_0)$, $f_{\varepsilon, i} = \xi (s (x) +
\tfrac{1}{i} - 2 \varepsilon)$, so
\begin{align}\begin{split}
& 1 + \tfrac{9}{4 \alpha^2} f_{\varepsilon, i}^2 - \tfrac{3}{2 \alpha^2} |
\nabla f_{\varepsilon, i} | \\
\geq & 1 + \tfrac{9}{4 \alpha^2} (\xi (s (x) + \tfrac{1}{i} - 2
\varepsilon))^2 - \xi' (s (x) + \tfrac{1}{i} - 2 \varepsilon) \\
\geq & 2 c \Lambda^{- 1} \tfrac{1}{\phi (s (x) + \tfrac{1}{i} - 2
\varepsilon)^2} \\
\geq & 2 c \Lambda^{- 1} \tfrac{1}{\phi (s (x))^2} .
\end{split}\end{align}
For $s (x) \in (r_0, r_1)$, we get
\begin{align}\begin{split}
& 1 + \tfrac{9}{4 \alpha^2} f_{\varepsilon, i}^2 - \tfrac{3}{2 \alpha^2} |
\nabla f_{\varepsilon, i} | \\
\geq & 1 + \tfrac{9}{4 \alpha^2} (\xi (h_{\varepsilon, i} (s)))^2 -
\tfrac{3}{2 \alpha^2} \xi' (h_{\varepsilon, i} (s)) h_{\varepsilon, i}' |
\nabla s | \\
\geq & 1 + \tfrac{9}{4 \alpha^2} (\xi (h_{\varepsilon, i} (s)))^2 -
\tfrac{3}{2 \alpha^2} \xi' (h_{\varepsilon, i} (s)) \tfrac{r_1 - (r_0 + i^{-
1} - 2 \varepsilon)}{r_1 - r_0} \\
\geq & 2 c \Lambda^{- 1} \tfrac{1}{\phi (h_{\varepsilon, i} (s))^2} -
\tfrac{3}{2 \alpha^2} \xi' (h_{\varepsilon, i} (s)) \tfrac{2 \varepsilon -
i^{- 1}}{r_1 - r_0} \\
\geq & 2 c \Lambda^{- 1} \tfrac{1}{\phi (s)^2} - C_{r_0, r_1}
\varepsilon.
\end{split}\end{align}
Let $u_{\varepsilon, i}$ be the solution to the spacetime harmonic function
boundary value problem below
\begin{align}\begin{split}
\Delta u_{\varepsilon, i} + 3 f_{\varepsilon, i} | \nabla u_{\varepsilon, i}
| & = 0 \text{ in } M_{\varepsilon}, \label{u e i spacetime harmonic}
\\
u_{\varepsilon, i} & = 1 \text{ on } \{t (x) = \tfrac{\pi}{b} -
\varepsilon\}, \\
u_{\varepsilon, i} & = - 1 \text{ on } \{t (x) = \varepsilon\} .
\end{split}\end{align}
\text{{\bfseries{Step 1}}}. We do modifications to the integral formula \eqref{eq:hkk integral formula}.

We apply \eqref{eq:hkk integral formula}, add a divergence term to its both
sides and obtain
\begin{align}\begin{split}
& \int_{\partial_- M_{\varepsilon}} | \nabla u_{\varepsilon, i} | \left(
\tfrac{3 (8 - c^{- 1})}{6 - c^{- 1}} f_{\varepsilon, i} - 2 H_{\varepsilon}
\right) - \int_{\partial_+ M_{\varepsilon}} | \nabla u_{\varepsilon, i} |
\left( \tfrac{3 (8 - c^{- 1})}{6 - c^{- 1}} f_{\varepsilon, i} + 2
H_{\varepsilon} \right) \\
= & \int_{\partial_- M_{\varepsilon}} | \nabla u_{\varepsilon, i} | (4
f_{\varepsilon, i} - 2 H_{\varepsilon}) - \int_{\partial_+ M_{\varepsilon}}
| \nabla u_{\varepsilon, i} | (4 f_{\varepsilon, i} + 2 H_{\varepsilon})
\\
& - \int_{M_{\varepsilon}} \tfrac{c^{- 1}}{6 - c^{- 1}}
\ensuremath{\operatorname{div}} (f_{\varepsilon, i} \nabla u_{\varepsilon,
i}) \\
\geq & \int_{M_{\varepsilon}} \left[ \frac{| \bar{\nabla}^2
u_{\varepsilon, i} |^2}{| \nabla u_{\varepsilon, i} |} + (R_g + 6
f_{\varepsilon, i}^2) | \nabla u_{\varepsilon, i} | - 4 \langle \nabla
u_{\varepsilon, i}, \nabla f_{\varepsilon, i} \rangle \right] - \int_{- 1}^1
4 \pi \chi (\Sigma_t) \\
& - \int_{M_{\varepsilon}} \tfrac{c^{- 1}}{6 - c^{- 1}}
\ensuremath{\operatorname{div}} (f_{\varepsilon, i} \nabla u_{\varepsilon,
i}) \\
= & \int_{M_{\varepsilon}} \left[ \frac{| \bar{\nabla}^2 u_{\varepsilon, i}
|^2}{| \nabla u_{\varepsilon, i} |} + (R_g + \tfrac{3 (12 - c^{- 1})}{6 - c^{-
1}} f_{\varepsilon, i}^2) | \nabla u_{\varepsilon, i} | - \tfrac{3 (8 - c^{-
1})}{6 - c^{- 1}} \langle \nabla u_{\varepsilon, i}, \nabla f_{\varepsilon,
i} \rangle \right] \\
& - \int_{- 1}^1 4 \pi \chi (\Sigma_t),
\end{split}\end{align}
where $\Sigma_t$ is the level set of $u_{\varepsilon, i}$.

On $\partial_- M_{\varepsilon}$, we have that $f_{\varepsilon, i} \to -
\infty$ as $i \to \infty$, choose any sufficiently large $i$ so that
\[ H_{\varepsilon} - \tfrac{3 (4 - c^{- 1})}{6 - c^{- 1}} f_{\varepsilon, i}
   \geq 0 \text{ on } \partial_- M_{\varepsilon}, \]
and
\[ H_{\varepsilon} \geq \tfrac{3 (4 - c^{- 1})}{6 - c^{- 1}}
   f_{\varepsilon, i} + \tfrac{12}{6 - c^{- 1}} f_{\varepsilon, i} = \tfrac{3
   (8 - c^{- 1})}{6 - c^{- 1}} f_{\varepsilon, i} \text{ on } \partial_-
   M_{\varepsilon} . \]
Similarly,
\[ \tfrac{3 (8 - c^{- 1})}{6 - c^{- 1}} f_{\varepsilon, i} + 2 H_{\varepsilon}
   \geq 0 \text{ on } \partial_+ M_{\varepsilon} . \]
These considerations lead to
\[ \int_{\partial_- M_{\varepsilon}} | \nabla u_{\varepsilon, i} | \left(
   \tfrac{3 (8 - c^{- 1})}{6 - c^{- 1}} f_{\varepsilon, i} - 2 H_{\varepsilon}
   \right) - \int_{\partial_+ M_{\varepsilon}} | \nabla u_{\varepsilon, i} |
   \left( \tfrac{3 (8 - c^{- 1})}{6 - c^{- 1}} f_{\varepsilon, i} + 2
   H_{\varepsilon} \right) \leq 0. \]
Hence,
\begin{align}\begin{split}
& \int_{- 1}^1 4 \pi \chi (\Sigma_t) \mathrm{d}t \\
\geq & \int_{M_{\varepsilon}} \left[ \frac{| \bar{\nabla}^2
u_{\varepsilon, i} |^2}{| \nabla u_{\varepsilon, i} |} + (R_g + \tfrac{3 (12 -
c^{- 1})}{6 - c^{- 1}} f_{\varepsilon, i}^2) | \nabla u_{\varepsilon, i} | -
\tfrac{3 (8 - c^{- 1})}{6 - c^{- 1}} \langle \nabla u_{\varepsilon, i},
\nabla f_{\varepsilon, i} \rangle \right] . \label{u e i integral
inequality}
\end{split}\end{align}
\text{{\bfseries{}}}\text{{\bfseries{Step 2}}}. Our goal is to construct a
function $\tilde{u}_{\varepsilon}$ whose gradient $| \nabla
\tilde{u}_{\varepsilon} |^{\tfrac{1}{2}}$ is compactly supported on
$M_{\varepsilon}$ out of the family $u_{\varepsilon, i}$ so that we could make
use of \eqref{L c}.

\

Using the construction of $f_{\varepsilon, i}$, we see that
\[ 1 + \tfrac{9}{4 \alpha^2} f_{\varepsilon, i}^2 - \tfrac{3}{2 \alpha^2} |
   \nabla f_{\varepsilon, i} | \geq - C_{r_0, r_1} \varepsilon . \]
We take $\varepsilon$ small so that
\[ 1 + \tfrac{9}{4 \alpha^2} f_{\varepsilon, i}^2 - \tfrac{3}{2 \alpha^2} |
   \nabla f_{\varepsilon, i} | \geq - 1. \]
 Since $\inf_M R_g >-\infty$, we assume that $R_g \geq -R_0$ for some positive constant $R_0$, so
\begin{align}\begin{split}
& R_g + \tfrac{3 (12 - c^{- 1})}{6 - c^{- 1}} f_{\varepsilon, i}^2 - \tfrac{3
(8 - c^{- 1})}{6 - c^{- 1}} | \nabla f_{\varepsilon, i} | \\
\geq & - R_{0} + \tfrac{3 c^{- 1}}{2 (6 - c^{- 1})} f_{\varepsilon, i}^2
- \tfrac{4 \alpha^2}{9} \cdot \tfrac{9 (8 - c^{- 1})}{2 (6 - c^{- 1})} - 1
\\
\geq & - R_{0} - c^{- 1} \Lambda - 1. \label{lower bound in Omega}
\end{split}\end{align}
Now we choose a region $\Omega = M_{\varepsilon_0}$ ($\varepsilon_0 > 2
\varepsilon$) sufficiently small depending only on $c$, $R_0$ and $\Lambda$
such that $| f_{\varepsilon, i} |$ is large enough that
\begin{equation}
  \tfrac{3 c^{- 1}}{2 (6 - c^{- 1})} f_{\varepsilon, i}^2 \geq
  \tfrac{c^{- 1} f_{\varepsilon, i}^2}{6 - c^{- 1}} + R_0 + c^{- 1} + 2
  \label{lower bound outside Omega}
\end{equation}
on $M_{\varepsilon} \backslash \Omega$ for sufficiently large $i$ and small
$\varepsilon$.

Fixing such an $\varepsilon_0$, we know that $f_{\varepsilon, i}$ are
uniformly bounded on $\Omega$. Using \eqref{lower bound in Omega} on $\Omega$
and \eqref{lower bound outside Omega} on $M_{\varepsilon} \backslash \Omega$
in the inequality \eqref{u e i integral inequality}, we have
\begin{align}\begin{split}
& \int_{- 1}^1 4 \pi \chi (\Sigma_t) \mathrm{d} t \\
\geq & \int_{M_{\varepsilon}} \left[ \frac{| \bar{\nabla}^2
u_{\varepsilon, i} |^2}{| \nabla u_{\varepsilon, i} |} + (R_g + \tfrac{3 (12 -
c^{- 1})}{6 - c^{- 1}} f_{\varepsilon, i}^2) | \nabla u_{\varepsilon, i} | \right]
\\
&\textbf{ }\quad-\int_{M_\epsilon} \tfrac{3 (8 - c^{- 1})}{6 - c^{- 1}} \langle \nabla u_{\varepsilon, i},
\nabla f_{\varepsilon, i}  \rangle \\
\geq & \int_{M_{\varepsilon}} \frac{| \bar{\nabla}^2 u_{\varepsilon, i}
|^2}{| \nabla u_{\varepsilon, i} |} - \int_{\Omega} (R_0 + c^{- 1} \Lambda +
1) | \nabla u_{\varepsilon, i} | + \int_{M\backslash \Omega} (\tfrac{c^{-
1}}{6 - c^{- 1}} f_{\varepsilon, i}^2 + 1) | \nabla u_{\varepsilon, i} | .
\label{after u e i lower bound applied}
\end{split}\end{align}
In order to extract a convergent subsequence, we rescale $u_{\varepsilon, i}$
and define
\[ \tilde{u}_{\varepsilon, i} = (\sup_{\Omega} | \nabla u_{\varepsilon, i}
   |)^{- 1} \left( u_{\varepsilon, i} - | \Omega |^{- 1} \int_{\Omega}
   u_{\varepsilon, i} \right) . \]
Obviously, $\sup_{\Omega} | \nabla \tilde{u}_{\varepsilon, i} | = 1$ and has
vanishing average on $\Omega$. Translating \eqref{after u e i lower bound
applied} using the definition of $\tilde{u}_{\varepsilon, i}$, we see that
\begin{align}\begin{split}
&\quad \tfrac{1}{\sup_{\Omega} | \nabla u_{\varepsilon, i} |} \int^1_{- 1} 4 \pi
\chi (\Sigma_t) \mathrm{d}t \\
\quad&\geq  \int_{M_{\varepsilon}} \frac{| \bar{\nabla}^2
\tilde{u}_{\varepsilon, i} |^2}{| \nabla \tilde{u}_{\varepsilon, i} |} -
\int_{\Omega} (R_0 + c^{- 1} \Lambda + 1) | \nabla \tilde{u}_{\varepsilon,
i} | + \int_{M_{\varepsilon} \backslash \Omega} (\tfrac{c^{- 1}}{6 - c^{-
1}} f_{\varepsilon, i}^2 + 1) | \nabla \tilde{u}_{\varepsilon, i} | .
\label{hkkz 2.11}
\end{split}\end{align}
We claim that $\sup_{\Omega} | \nabla u_{\varepsilon, i} |$ is uniformly
bounded below, otherwise we ended up with a trivial inequality. We defer the
proof of this claim to a later Lemma \ref{lower gradient bound}.

Since $| \nabla \tilde{u}_{\varepsilon, i} | \leq 1$ on $\Omega$, we see
then
\begin{align}\begin{split}
& \int_{M_{\varepsilon} \backslash \Omega} (\tfrac{c^{- 1}}{6 - c^{- 1}}
f_{\varepsilon, i}^2 + 1) | \nabla \tilde{u}_{\varepsilon, i} | \\
\leq & C_1 \int_{- 1}^1 4 \pi \chi (\Sigma_t) \mathrm{d}t + \int_{\Omega} (R_{0} +
c^{- 1} \Lambda + 1) | \nabla \tilde{u}_{\varepsilon, i} | \\
\leq & C_1 \int_{- 1}^1 4 \pi \chi (\Sigma_t) \mathrm{d}t + (R_0 + c^{- 1} \Lambda
+ 1) | \Omega | .
\end{split}\end{align}
Since $f_{\varepsilon, i}$ is uniformly bounded on $\Omega$, we have
\begin{align}\begin{split}
& \int_{M_{\varepsilon}} (f_{\varepsilon, i}^2 + 1) | \nabla
\tilde{u}_{\varepsilon, i} | \\
= & \left( \int_{M_{\varepsilon} \backslash \Omega} + \int_{\Omega} \right)
(f_{\varepsilon, i}^2 + 1) | \nabla \tilde{u}_{\varepsilon, i} | \\
\leq & \left( \tfrac{6 - c^{- 1}}{c^{- 1}} + 1 \right) \left( C_1
\int_{- 1}^1 4 \pi \chi (\Sigma_t) \mathrm{d}t + (R_0 + c^{- 1} \Lambda + 1) | \Omega |
\right) \\
& + | \Omega | \sup_{\Omega} (| f_{\varepsilon, i} |^2 + 1) \\
\leq & C_2, \label{hkkz 2.14}
\end{split}\end{align}
where $C_2$ is independent of $\varepsilon$ and $i$. 

Now we can follow
{\cite{hirsch-spectral-2023-arxiv}} to show that there exists a suitable limit
$\tilde{u}_{\varepsilon}$ of the sequence $\tilde{u}_{\varepsilon, i}$ when $i
\to \infty$ with the crucial property that $\nabla \tilde{u}_{\varepsilon}$ is
supported in $M_{\varepsilon}$ (in fact $M_{2 \varepsilon}$) with
integrability required by the definition of $\Lambda_c$ in \eqref{L c}. For completeness, we repeat their proof here.

Since $\tilde{u}_{\varepsilon, i}$ is of zero average in $\Omega$, we apply a
version of the Poincar{\'e} inequality on $M_{\varepsilon}$ (see
{\cite[Theorem 1]{lieb-poincare-2003}}) to conclude that $\|
\tilde{u}_{\varepsilon, i} \|_{W^{1, 1} (M_{\varepsilon})}$ is bounded below
by a constant independent of $i$. Therefore, by passing to a subsequence (in
$i$), $\tilde{u}_{\varepsilon, i}$ converges to a function
$\tilde{u}_{\varepsilon}$ in $L^p (M_{\varepsilon})$ for $p \in [1,
\tfrac{3}{2})$ as $i \to \infty$. because that $\tilde{u}_{\varepsilon, i}$
solves the elliptic spacetime Laplace equation uniform $L^p (M_{\varepsilon})$
bounds for $\tilde{u}_{\varepsilon, i}$ imply uniform control in $C^{2,
\alpha}_{\ensuremath{\operatorname{loc}}} (M_{2 \varepsilon})$ for some
$\alpha \in (0, 1)$; here we have used the fact that $f_{\varepsilon, i} \to
f_{\varepsilon}$ pointwise on the interior of $M_{2 \varepsilon}$ as $i \to
\infty$. Thus, $\tilde{u}_{\varepsilon, i}$ also converges subsequentially as
$i \to \infty$ to $\tilde{u}_{\varepsilon}$ in $C^{2, \alpha} (M_{2
\varepsilon})$ for some $\alpha \in (0, 1)$, and the limit satisfies $\Delta
\tilde{u}_{\varepsilon} + 3 f_{\varepsilon} | \nabla \tilde{u}_{\varepsilon} |
= 0$ in the interior of $M_{2 \varepsilon}$.

For the limit $\tilde{u}_{\varepsilon}$ we have the following crucial
property.

\begin{lemma}
  For $\tilde{u}_{\varepsilon}$ constructed above, we have that
  \[ | \nabla \tilde{u}_{\varepsilon} |^{\tfrac{1}{2}} \in H_0^1 (M_{2
     \varepsilon}) \]
  for almost all $\varepsilon > 0$ sufficiently small.
\end{lemma}

\begin{proof}
  Using that $c > \tfrac{1}{6}$ and \eqref{hkkz 2.11}, we see
\begin{align}\begin{split}
& \tfrac{1}{\sup_{\Omega} | \nabla u_{\varepsilon, i} |} \int^1_{- 1} 4
\pi \chi (\Sigma_t) \mathrm{d}t + \int_{\Omega} (R_0 + c^{- 1} \Lambda + 1) | \nabla
\tilde{u}_{\varepsilon, i} | \\
\geq & \int_{M_{\varepsilon}} \frac{| \bar{\nabla}^2
\tilde{u}_{\varepsilon, i} |^2}{| \nabla \tilde{u}_{\varepsilon, i} |}
\\
= & \int_{M_{\varepsilon}} \frac{| \bar{\nabla}^2 \tilde{u}_{\varepsilon,
i} + f_{\varepsilon, i} | \nabla \tilde{u}_{\varepsilon, i} |g|^2}{|
\nabla \tilde{u}_{\varepsilon, i} |} \\
\geq & \int_{M_{\varepsilon}} \frac{| \bar{\nabla}^2
\tilde{u}_{\varepsilon, i} |^2}{| \nabla \tilde{u}_{\varepsilon, i} |} - 3
f_{\varepsilon, i}^2 | \nabla \tilde{u}_{\varepsilon, i} | \\
\geq & 4 \int_{M_{\varepsilon}} | \nabla | \nabla
\tilde{u}_{\varepsilon, i} |^{\tfrac{1}{2}} |^2 - 3 C_2,
\end{split}\end{align}
  where $C_2$ is as in \eqref{hkkz 2.14}. Using Lemma \ref{lower gradient
  bound}, we conclude that $| \nabla \tilde{u}_{\varepsilon, i}
  |^{\tfrac{1}{2}}$ is bounded in $H^1 (M_{\varepsilon})$ independent of $i$,
  it has a weak subsequential limit in $H^1 (M_{\varepsilon})$ which also
  converges strongly in $L^2 (M_{\varepsilon})$.
  
  Since that $f_{\varepsilon, i}$ blows up uniformly on $M_{\varepsilon}
  \backslash M_{2 \varepsilon}$, the inequality \eqref{hkkz 2.14} implies that
  the limit function $\tilde{u}_{\varepsilon}$ satisfies $| \nabla
  \tilde{u}_{\varepsilon} |^{\tfrac{1}{2}} = 0$ almost everywhere on
  $M_{\varepsilon} \backslash M_{2 \varepsilon}$. Note that even though
  $\tilde{u}_{\varepsilon}$ may not be defined on $M_{\varepsilon} \backslash
  M_{2 \varepsilon}$, with a slight abuse of notation we still denote the
  limit of $| \nabla \tilde{u}_{\varepsilon, i} |^{\tfrac{1}{2}}$ in terms of
  $\tilde{u}_{\varepsilon}$. Moreover, for almost every $\varepsilon$, the
  boundaries $\partial_{\pm} M_{2 \varepsilon}$ is a Lipschitz manifold and
  therefore by the trace theorem, we have that $| \nabla
  \tilde{u}_{\varepsilon} |^{\tfrac{1}{2}}$ vanishes up to a set of measure
  zero on this set. It follows that $| \nabla \tilde{u}_{\varepsilon}
  |^{\tfrac{1}{2}} \in H_0^1 (M_{2 \varepsilon})$ for almost all sufficiently
  small $\varepsilon > 0$.
\end{proof}

Now we turn to the integral inequality \eqref{u e i integral inequality}.
Taking the limit as $i \to \infty$ and applying Fatou's Lemma to \eqref{u e i
integral inequality} yields
\begin{align}\begin{split}
& 4 \pi \int_{\tilde{u}_{\varepsilon}^-}^{\tilde{u}_{\varepsilon}^+} \chi
(\Sigma_t) \mathrm{d} t \\
\geq & \int_{M_{2 \varepsilon}} \frac{| \nabla^2
\tilde{u}_{\varepsilon} + f_{\varepsilon} | \nabla \tilde{u}_{\varepsilon}
|g|^2}{| \nabla \tilde{u}_{\varepsilon} |} + | \nabla
\tilde{u}_{\varepsilon} |  \left( R_g + \frac{3 (12 - c^{- 1})}{6 - c^{- 1}}
f_{\varepsilon}^2 \right) \\
& - \int_{M_{2 \varepsilon}} \frac{3 (8 - c^{- 1})}{6 - c^{- 1}}  \langle
\nabla \tilde{u}_{\varepsilon}, \nabla f_{\varepsilon} \rangle .
\label{integral u tilde epsilon}
\end{split}\end{align}
Here $\tilde{u}_{\varepsilon}^{\pm}$ are (subsequential) limit of
$\tilde{u}_{\varepsilon, i}$ on $\partial_{\pm} M_{\varepsilon}$ as $i \to
\infty$. Note that $| \nabla \tilde{u}_{\varepsilon} |^{\tfrac{1}{2}}$
vanishes almost everywhere on $M_{\varepsilon} \backslash M_{2 \varepsilon}$,
so the constants $\tilde{u}_{\varepsilon}^{\pm}$ are actually also limits of
$\tilde{u}_{\varepsilon, i}$ on $\partial_{\pm} M_{2 \varepsilon}$. And they
are uniformly bounded independent of $\varepsilon$ due to Lemma \ref{lower
gradient bound}.

We estimate the first two terms on the above inequality. Recall that we have
\begin{align}\begin{split}
& | \nabla^2  \tilde{u}_{\varepsilon} + f_{\varepsilon} | \nabla
\tilde{u}_{\varepsilon} |g|^2 - \tfrac{3}{2} (3| \nabla | \nabla
\tilde{u}_{\varepsilon} | + f_{\varepsilon} \nabla \tilde{u}_{\varepsilon}
|^2) \\
\geq & \tfrac{1}{4} (\nabla_{1 1} \tilde{u}_{\varepsilon} - \nabla_{2
2} \tilde{u}_{\varepsilon})^2 + \tfrac{1}{2} (| \nabla_{1 2}
\tilde{u}_{\varepsilon} |^2 + | \nabla_{13} \tilde{u}_{\varepsilon} |^2 + |
\nabla_{2 3} \tilde{u}_{\varepsilon} |^2) \label{remark 4.4}
\end{split}\end{align}
from {\cite[Remark 4.4]{hirsch-rigid-2022-arxiv}} where $\{e_1, e_2 \}$
represents an orthonormal frame on the level set of $u$ and $e_3 = | \nabla
\tilde{u}_{\varepsilon} |^{- 1} \nabla \tilde{u}_{\varepsilon}$. So
\begin{align}\begin{split}
& \int_{M_{2 \epsilon}} \frac{| \nabla^2  \tilde{u}_{\varepsilon} +
f_{\varepsilon} | \nabla \tilde{u}_{\varepsilon} |g|^2}{| \nabla
\tilde{u}_{\varepsilon} |} + R_g | \nabla \tilde{u}_{\varepsilon} |
\\
\geq & \int_{M_{2 \epsilon}} \frac{3 | \nabla | \nabla
\tilde{u}_{\varepsilon}  | + f_{\varepsilon} \nabla \tilde{u}_{\varepsilon}
|^2}{2 | \nabla \tilde{u}_{\varepsilon} |} + R_g | \nabla
\tilde{u}_{\varepsilon} | \\
= & \int_{M_{2 \epsilon}} 6 | \nabla | \nabla \tilde{u}_{\varepsilon}
|^{\frac{1}{2}} |^2 + 6 \langle \nabla | \nabla \tilde{u}_{\varepsilon}
|^{\frac{1}{2}}, f_{\varepsilon} | \nabla \tilde{u}_{\varepsilon} |^{-
\frac{1}{2}} \nabla \tilde{u}_{\varepsilon} \rangle + \left( R_g + \frac{3}{2}
f_{\varepsilon}^2 \right)  | \nabla \tilde{u}_{\varepsilon} | \\
= & \int_{M_{2 \epsilon}} (6 - c^{- 1})  \left| \nabla | \nabla
\tilde{u}_{\varepsilon} |^{\frac{1}{2}} + \frac{3}{6 - c^{- 1}}
f_{\varepsilon} | \nabla \tilde{u}_{\varepsilon} |^{- \frac{1}{2}} \nabla
\tilde{u}_{\varepsilon} \right|^2 + c^{- 1} | \nabla | \nabla
\tilde{u}_{\varepsilon} |^{\frac{1}{2}} |^2 \\
& + \int_{M_{2 \varepsilon}} R_g | \nabla \tilde{u}_{\varepsilon} | + \left(
\frac{3}{2} - \frac{9}{6 - c^{- 1}} \right) f_{\varepsilon}^2 | \nabla
\tilde{u}_{\varepsilon} | \\
\geq & \int_{M_{2 \epsilon}} c^{- 1} \Lambda_c  | \nabla
\tilde{u}_{\varepsilon} | + \left( \frac{3}{2} - \frac{9}{6 - c^{- 1}}
\right) f_{\varepsilon}^2 | \nabla \tilde{u}_{\varepsilon} | .
\end{split}\end{align}
In the last inequality, we have used \eqref{L c}. Rewriting the above
slightly, we see that
\begin{align}\begin{split}
& \int_{M_{2 \varepsilon}} \frac{| \nabla^2  \tilde{u}_{\varepsilon} +
f_{\varepsilon} | \nabla \tilde{u}_{\varepsilon} |g|^2}{| \nabla
\tilde{u}_{\varepsilon} |} + R_g | \nabla \tilde{u}_{\varepsilon} |
\\
\geq & \int_{M_{2 \varepsilon}} c^{- 1} \Lambda | \nabla
\tilde{u}_{\varepsilon} | + \left( \frac{3}{2} - \frac{9}{6 - c^{- 1}}
\right) f_{\varepsilon}^2 | \nabla \tilde{u}_{\varepsilon} | + \int_{M_{2
\varepsilon}} c^{- 1} (\Lambda_c - \Lambda) | \nabla \tilde{u}_{\varepsilon}
| .
\end{split}\end{align}
Using the above estimate in \eqref{integral u tilde epsilon}, we arrive that
\begin{align}\begin{split}
& 4 \pi \int_{\tilde{u}_{\varepsilon}^-}^{\tilde{u}_{\varepsilon}^+} \chi
(\Sigma_t) \mathrm{d} t \\
\geq & \int_{M_{2 \varepsilon}} \left( c^{- 1} \Lambda + \frac{9 (8 -
c^{- 1})}{2 (6 - c^{- 1})} f_{\varepsilon}^2 \right) | \nabla
\tilde{u}_{\varepsilon} | - \frac{3 (8 - c^{- 1})}{6 - c^{- 1}}  \langle
\nabla \tilde{u}_{\varepsilon}, \nabla f_{\varepsilon} \rangle \\
& + \int_{M_{2 \varepsilon}} c^{- 1} (\Lambda_c - \Lambda) | \nabla
\tilde{u}_{\varepsilon} | \\
& + \int_{M_{2 \epsilon}} (6 - c^{- 1})  \left| \nabla | \nabla
\tilde{u}_{\varepsilon} |^{\frac{1}{2}} + \frac{3}{6 - c^{- 1}}
f_{\varepsilon} | \nabla \tilde{u}_{\varepsilon} |^{- \frac{1}{2}} \nabla
\tilde{u}_{\varepsilon} \right|^2 \\
\geq & \int_{M_{2 \varepsilon}} c^{- 1} \Lambda \left( 1 + \frac{9}{4
\alpha^2} f_{\varepsilon}^2 - \frac{3}{2 \alpha^2} | \nabla f_{\varepsilon}
| \right)  | \nabla \tilde{u}_{\varepsilon} | \\
& + \int_{M_{2 \varepsilon}} c^{- 1} (\Lambda_c - \Lambda) | \nabla
\tilde{u}_{\varepsilon} | \\
& + \int_{M_{2 \varepsilon}} (6 - c^{- 1})  \left| \nabla | \nabla
\tilde{u}_{\varepsilon} |^{\frac{1}{2}} + \frac{3}{6 - c^{- 1}}
f_{\varepsilon} | \nabla \tilde{u}_{\varepsilon} |^{- \frac{1}{2}} \nabla
\tilde{u}_{\varepsilon} \right|^2, \label{hkkz 2.18}
\end{split}\end{align}
where $\alpha = \sqrt{\frac{\Lambda (6 - c^{- 1})}{2 c (8 - c^{- 1})}}$.

Before we turn to \text{{\bfseries{Step 3}}} of the proof, we give a lemma
used in \text{{\bfseries{Step 2}}}.

\begin{lemma}
  \label{lower gradient bound}Let $u_{\varepsilon, i}$ be as \eqref{u e i
  spacetime harmonic} and $\Omega$ be as above, then
  \[ \sup_{\Omega} | \nabla u_{\varepsilon, i} | \geq \tfrac{1}{C_1} > 0,
  \]
  where $C_1$ does not depend on $\varepsilon$ and $i$.
\end{lemma}

\begin{proof}
  We consider a region $M_{\varepsilon'}$ slightly larger than $\Omega =
  M_{\varepsilon_0}$. That is, we pick an $\varepsilon'$ sufficiently close to
  $\varepsilon_0$ with $2 \varepsilon < \varepsilon' < \varepsilon_0$. By the
  maximum principle, $|u_{\varepsilon, i} | \leq 1$ \ in
  $M_{\varepsilon}$. Then by standard interior elliptic estimates in
  $M_{\varepsilon'}$, we have that $u_{\varepsilon, i}$ are in $C^{2, \mu}
  (\bar{\Omega})$ for some $\mu \in (0, 1)$. We take a geodesic segment
  $\gamma (t)$ in $\bar{\Omega}$ realizing the distance between two boundary
  components of $\Omega$ with $t \in [0, t_0]$ where $t_0
  =\ensuremath{\operatorname{dist}} (\partial_- \Omega, \partial_+ \Omega)$.
  
  For some $t' \in [0, t_0]$,
  \[ \tfrac{u (\gamma (t_0)) - u (\gamma
     (0))}{\ensuremath{\operatorname{dist}} (\partial_- \Omega, \partial_+
     \Omega)} = \tfrac{u (\gamma (t_0)) - u (\gamma (0))}{t_0 - 0} =
     \tfrac{\mathrm{d}}{\mathrm{d} t} u_{\varepsilon, i} (\gamma (t)) |_{t =
     t'}. \]
  Since also $\tfrac{\mathrm{d}}{\mathrm{d} t} u_{\varepsilon, i} (\gamma (t))
  = \langle \nabla u_{\varepsilon, i}, \gamma' \rangle$ for every $t \in [0,
  t_0]$ and that $\gamma$ is of unit speed, we then conclude that
  \[ \sup_{\Omega} | \nabla u_{\varepsilon, i} | \geq
     \tfrac{\mathrm{d}}{\mathrm{d} t} u_{\varepsilon, i} (\gamma (t)) |_{t =
     t'} = \tfrac{u (\gamma (t_0)) - u (\gamma
     (0))}{\ensuremath{\operatorname{dist}} (\partial_- \Omega, \partial_+
     \Omega)} . \]
  If we can show that $u_{\varepsilon, i} \geq \tfrac{1}{2}$ on
  $\partial_+ \Omega$ and $u_{\varepsilon, i} \leq - \tfrac{1}{2}$ on
  $\partial_- \Omega$, we are done.

To proceed, we construct a supersolution near $\{s = 0\}$. Let
\[ w := w_{\varepsilon, i} = - (1 + l \varepsilon) + ls \]
where $l$ is a positive constant to be determined. The function $w$ agrees
with $u_{\varepsilon, i}$ on $\{s = \varepsilon\}$. For some number $r_1 \in
(2 \varepsilon, r_0)$, on the set $\{s = r_1 \}$, we have
\[ w = - 1 - l \varepsilon + l r_1 . \]
We choose $l = \tfrac{3}{r_1}$, then $w > 2 - l \varepsilon > 1$ on $\{s = r_1
\}$ if $\varepsilon$ is sufficiently small. For a function $u$ of $s$,
\begin{align}
  & \Delta w + 3 f_{\varepsilon, i} | \nabla w| \\
  = & \tfrac{1}{\sqrt{g}} \sum_{i, j = 1}^3 \partial_i  (g^{ij}  \sqrt{g}
  \partial_j w) + 3 f_{\varepsilon, i}  \sqrt{g^{s s}} l \\
  = & \tfrac{1}{\sqrt{g}} \sum_{i = 1}^3 \partial_i  (g^{is}  \sqrt{g}) w' +
  g^{ss} w'' + 3 f_{\varepsilon, i}  \sqrt{g^{s s}} l \\
  = & l (\Delta s + 3 f_{\varepsilon, i} \sqrt{g^{s s}} l) . 
\end{align}
By construction of $f_{\epsilon,i}$, for sufficiently large $i$, we see $f_{\varepsilon, i}
\leqslant f$, so
\[ \Delta w + 3 f_{\varepsilon, i} | \nabla w| \leqslant l (\Delta s + 3 f
   \sqrt{g^{s s}}) \leqslant - \tfrac{c_1 l}{s} + O (1)\]
   on $[\epsilon, r_1]$ due to \eqref{technical0}.
We see that $w$ is a supersolution on $\{x \in M : \varepsilon \leqslant s \leqslant r_1 \}$ choosing $r_1$ and $\epsilon$ small, so by maximum principle $u_{\varepsilon, i} \leqslant w$ on \(\epsilon \leq s \leq r_1\). In particular on $\partial_- \Omega = \{ s=\epsilon_0\}$, we have
\[ u_{\varepsilon, i} \leqslant - 1 - l \varepsilon + l \varepsilon_0 . \]
Now choosing smaller $\varepsilon$ and $\varepsilon_0$, we arrive that
$u_{\varepsilon, i} \leqslant - \tfrac{1}{2}$ on $\partial_- \Omega$. We
follow the same procedures, we can arrive that $u_{\varepsilon, i} \geqslant
\tfrac{1}{2}$ on $\partial_+ \Omega$.
\end{proof}

\

\text{{\bfseries{Step 3}}}. Limiting behavior of $\tilde{u}_{\varepsilon}$ as
$\varepsilon \to 0$.

\

To proceed, we shall investigate the limit of $\tilde{u}_{\varepsilon}$ as
$\varepsilon \to 0$. First note that applying Fatou's lemma to equation
\eqref{hkkz 2.14} yields
\[ \int_{M_{2 \varepsilon}} | \nabla \tilde{u}_{\varepsilon} | \leq
   \liminf_{i \to \infty}  \int_{M_{2 \varepsilon}} | \nabla
   \tilde{u}_{\varepsilon, i} | \leq C_1 . \]
Since $\tilde{u}_{\varepsilon, i}$ has vanishing average on $\Omega$, the same
holds true for $\tilde{u}_{\varepsilon}$, and thus utilizing again the
Poincar{\'e} inequality {\cite[Theorem 1]{lieb-poincare-2003}} we obtain
uniform $W^{1, 1} (M_{2 \varepsilon})$ bounds for $\tilde{u}_{\varepsilon}$.
By passing to a subsequence, $\tilde{u}_{\varepsilon} \to u$ in
$L^p_{\ensuremath{\operatorname{loc}}} (M)$ for any $p \in [1, \frac{3}{2})$
for some $u \in L^p_{\ensuremath{\operatorname{loc}}} (M)$ where $M =
\cup_{\varepsilon} M_{2 \varepsilon}$.

As before, since $\tilde{u}_{\varepsilon}$ satisfies the elliptic spacetime
Laplacian, we may bootstrap to find subsequential convergence
$\tilde{u}_{\varepsilon} \to u$ in $C^{2,
\mu}_{\ensuremath{\operatorname{loc}}} (M )$, for some $\mu \in (0, 1)$.
Furthermore, $\Delta u + 3 f | \nabla u| = 0$ on $M$ with
\[ f (x) = \xi (s (x)) \label{definition of f} \]
using the definition as in \eqref{f e j}. Again using that $g \geq
\bar{g}$ similarly as in \eqref{gradient f e i}, we see that
\[ - | \nabla f| \geq - \xi' \]
which leads to
\begin{equation}
  1 + \frac{9}{4 \alpha^2} f^2 - \frac{3}{2 \alpha^2}  | \nabla f| \geq 1
  + \tfrac{9}{4 \alpha^2} \xi^2 - \tfrac{3}{2 \alpha^2} \xi' = 2 c \Lambda^{-
  1} \phi^{- 2} . \label{ode variant}
\end{equation}
Furthermore, taking the limit of \eqref{hkkz 2.18} with Fatou's lemma implies
that
\begin{align}\begin{split}
& 4 \pi \int^{u^+}_{u_-} \chi (\Sigma_t) \mathrm{d} t \\
\geq & \int_M c^{- 1} \Lambda \left( 1 + \frac{9}{4 \alpha^2} f^2 -
\frac{3}{2 \alpha^2} | \nabla f| \right)  | \nabla u| \\
& + \int_M c^{- 1} (\Lambda_c - \Lambda) | \nabla u| \\
& + \int_M (6 - c^{- 1})  \left| \text{ } \nabla | \nabla u|^{\frac{1}{2}}
+ \frac{3}{6 - c^{- 1}} f| \nabla u|^{- \frac{1}{2}} \nabla u \right|^2,
\end{split}\end{align}
where $u^{\pm}$ are subsequential limits of $\tilde{u}_{\varepsilon}^{\pm}$ on
$\partial_{\pm} M$ and $\Sigma_t = \{x \in M : u (x) = t\}$. We have
established \eqref{ode variant}, with $\Lambda > 0$, we see
\begin{align}\begin{split}
& 4 \pi \int^{u^+}_{u_-} \chi (\Sigma_t) \mathrm{d} t \\
\geq & 2 \int_M \phi^{- 2} | \nabla u| + \int_M c^{- 1} (\Lambda_c -
\Lambda) | \nabla u| \\
& + \int_M (6 - c^{- 1})  \left| \text{ } \nabla | \nabla u|^{\frac{1}{2}}
+ \frac{3}{6 - c^{- 1}} f| \nabla u|^{- \frac{1}{2}} \nabla u \right|^2 .
\end{split}\end{align}
Applying a linear function on $u$ we may assume that $u^+ = 1$ and $u_- = -
1$, we finish the proof of Theorem \ref{final integral estimate}.

\

\

\subsection{Proof of Theorem \ref{harmonic main}}

Comparing with Theorem \ref{final integral estimate}, only the condition
\[ \Lambda_c \geq \Lambda > 0 \]
is imposed in Theorem \ref{harmonic main}.

\begin{proof}[Proof of Theorem \ref{harmonic main}]
  Following from Theorem \ref{final integral estimate} and $\Lambda_c
  \geq \Lambda > 0$, we obtain that
  \begin{equation}
    4 \pi \int_{- 1}^1 \chi (\Sigma_t) \mathrm{d} t \geq \int_M 2 \phi^{-
    2} | \nabla u| \mathrm{d} v \geq \int_M 2 \phi^{- 2} | \nabla u|
    \mathrm{d} \bar{v} . \label{apply lower bound on eigenvalue}
  \end{equation}
  Here $\mathrm{d} v$ and $\mathrm{d} \bar{v}$ be the three dimensional
  Hausdorff measure with respect to the metrics $g$ and $\bar{g}$ which we
  have to write it explicitly. In the second inequality, we used $g \geq
  \bar{g}$.
  
  We show in the following that the above inequality is actually an
  equality. Using the coarea formula on the right, we see that
  \[ 4 \pi \int_{- 1}^1 \chi (\Sigma_t) \mathrm{d} t \geq \int_{- 1}^1
     \left( \int_{\Sigma_t} 2 \phi^{- 2} \mathrm{d} \bar{\sigma} \right)
     \mathrm{d} t, \]
  where $\mathrm{d} \bar{\sigma}$ is the two dimensional Hausdorff measure on
  $\Sigma_t$ with respect to the induced metric from $\bar{g}$.
  
  Let $p$ be the projection of $\Sigma_t$ to the $\mathbb{S}^2$ factor in $M
  =\mathbb{S}^2 \times (0, \tfrac{\pi}{b})$. Since $\Sigma_t$ is the $t$-level set of the function $u$ which is the limit
of functions taking constant values in two level sets of $s$, so $\Sigma_t$
must have a connected component which intersects $\{(z,s)\in  \mathbb{S}^2\times(0, \tfrac{\pi}{b})\}$ for each fixed $z$. We now rename $\Sigma_t$ to be this component, then for $x \in \Sigma_t$ is
  represented by $(p (x), s (x))$ and we may parametrize each $\Sigma_t$ by
  $\mathbb{S}^2$ by
  \[ (z, s (x)), p (x) = z \in \mathbb{S}^2, x \in \Sigma_t . \]
  Locally we can represent a small neighborhood of $x \in \Sigma_t$ as a graph
  over $\mathbb{S}^2$ near $z$ in $M =\mathbb{S}^2 \times (0,
  \tfrac{\pi}{b})$. Hence
  \[ \int_{\Sigma_t} \phi^{- 2} (s (x)) \mathrm{d} \bar{\sigma} \geq
     \int_{\mathbb{S}^2} \phi^{- 2} \phi^2 \mathrm{d}
     \bar{\sigma}_{\mathbb{S}^2} = \int_{\mathbb{S}^2} \mathrm{d}
     \bar{\sigma}_{\mathbb{S}^2} = 2 \pi, \]
  where $\mathrm{d} \bar{\sigma}_{\mathbb{S}^2}$ is the standard two
  dimensional Hausdorff measure on $\mathbb{S}^2$ and
  \[ 4 \pi \int_{- 1}^1 \chi (\Sigma_t) \mathrm{d} t \geq \int_{- 1}^1 4
     \pi \mathrm{d} t. \]
  Since we have the simple fact that $\chi (\Sigma_t) \leq 2 = \chi
  (\mathbb{S}^2)$ for all $t$, so the inequality \eqref{apply lower bound on
  eigenvalue} is actually an equality.
  
  Following directly from \eqref{final crucial integral inequality}, we see
  that
  \begin{equation}
    \text{ } \nabla | \nabla u|^{\frac{1}{2}} + \frac{3}{6 - c^{- 1}} f |
    \nabla u|^{- \frac{1}{2}} \nabla u = 0 \label{excess of second order}
  \end{equation}
  holds almost everywhere on $M$. Integrating this equation along curves
  starting from regular points of $u$ shows that $\nabla u$ is nonvanishing
  globally on $M$.
  
  Let $e_3 = | \nabla u|^{- 1} \nabla u$ and we extend $e_3$ to an orthonormal
  frame $\{e_1, e_2, e_3 \}$. So we see that $e_1$ and $e_2$ are parallel to
  level sets of $u$.
  
  Tracing back the proof in Theorem \ref{final integral estimate} (taking
  limits in appropriate places), we find that the inequalities
  \begin{equation}
    - | \nabla f| \geq - \xi', - \langle \nabla f, \nabla u \rangle
    \geq - | \nabla f|  | \nabla u| \label{f u relation}
  \end{equation}
  are also in fact equalities. Since $f = \xi (s)$, we know then that $|
  \nabla s| = 1$ and $u$ is a function of $s$. So every level set of $u$ is
  just a coordinate 2-sphere.
  
  We also have that after taking limits the inequality \eqref{remark 4.4} is
  an equality which gives
  \begin{equation}
    \nabla_{1 1} u = \nabla_{2 2} u, \nabla_{1 2} u = \nabla_{1 3} u =
    \nabla_{2 3} u = 0. \label{hessian of u}
  \end{equation}
  Together with the fact that $| \nabla s| = 1$, we know that $(M, g)$ is a
  warped product with
  \begin{equation}
    g = \mathrm{d}s^2+ \varphi (s)^2 g_0  \label{preliminary metric}
  \end{equation}
  where $g_0$ is some metric on $\mathbb{S}^2$ and $\varphi$ is a positive
  continuously differentiable function on $(0, \tfrac{\pi}{b})$.
  
  We note that $f$ is monotonically increasing as $s$ increases, the second
  inequality of \eqref{f u relation} being an equality implies that $u$ is
  also monotonically increasing as $s$ increases. So
  \[ u' (s) = \partial_s u = \nabla_s u = \nabla_3 u = | \nabla u|, \]
  and from \eqref{excess of second order}, we conclude that
  \begin{equation}
    \nabla_3 | \nabla u| = \partial_s | \nabla u| = - \tfrac{6}{6 - c^{- 1}} f
    | \nabla u| . \label{s derivative of gradient u}
  \end{equation}
  This ordinary differential equation of $| \nabla u|$ is easily solvable as
  follows
  \[ u' (s) = | \nabla u| = \exp \left(- \tfrac{6}{6 - c^{- 1}} \int^s f (\tau)
     \mathrm{d} \tau\right) + c_0 \]
   where $c_0$ is a constant such that $u'(0)=0$. Since $f= \tfrac{2(6-c^{-1})}{3(c^{-1}-4)}\phi^{-1}\phi'$, we see then $|\nabla u| = c_1 \phi^{4/(4-c^{-1})}$ for some positive constant $c_1$.
   It follows from \eqref{s derivative of
  gradient u} and that
  \begin{equation}
    \nabla_{3 3} u = \tfrac{1}{| \nabla u|} \langle \nabla_3 \nabla u, \nabla
    u \rangle = \nabla_3 | \nabla u| = - \tfrac{6}{6 - c^{- 1}} f | \nabla u|
    . \label{hess 33}
  \end{equation}
  Let $H$ be the mean curvature of each level set with respect to the unit
  normal $e_3 = \partial_s$. By \eqref{preliminary metric}, $H = 2 \varphi^{-
  1} \varphi'$ and so
  \[ H = \tfrac{1}{| \nabla u|} (\Delta u - \nabla_{3 3} u) = \tfrac{3 c^{- 1}
     - 12}{6 - c^{- 1}} f = 2 \varphi^{- 1} \varphi' \]
  where we have used \eqref{hess 33} and that $u$ is spacetime harmonic
  function. By the definition of $f$ in \eqref{definition of f}, $\varphi =
  \phi$.
  
  Now it remains to show that $g_0$ is the standard round metric.
  
  First, we calculate the scalar curvature $R_g$ of $(M, g)$ in terms of
  $\phi$. Using the first variation of the mean curvature in the direction
  $\partial_s$, we see that
  \[ H' = -\ensuremath{\operatorname{Ric}} (\partial_s) - |A|^2 \]
  where $A$ is the second fundamental form of the level set. Using the classic
  Schoen-Yau rewrite (essentially traces of the Gauss equation; see
  {\cite{schoen-proof-1979}}), we know that
\begin{align}\begin{split}
- 2 H' & = R_g - 2 \phi^{- 2} K + |A|^2 + H^2,
\end{split}\end{align}
  where $K$ is the Gauss curvature of $(\mathbb{S}^2, g_0)$. Recall that $(M,
  g)$ is a warped product, so $|A|^2 = \tfrac{1}{2} H^2$. Hence,
  \begin{equation}
    R_g = 2 \phi^{- 2} K - 2 (\tfrac{2 \phi'}{\phi})' - \tfrac{3}{2} (\tfrac{2
    \phi'}{\phi})^2 . \label{scalar curvature with unkown K}
  \end{equation}
  Since $(M, g)$ is a warped product, so the Laplacian $\Delta$ for functions
  of $s$ is quite easy to calculate. The function $| \nabla u|^{\tfrac{1}{2}}$
  is a function of $s$, so
  \[ \Delta | \nabla u|^{\tfrac{1}{2}} = (| \nabla u|^{\tfrac{1}{2}})'' +
     \tfrac{2 \phi'}{\phi} (| \nabla u|^{\tfrac{1}{2}})', \]
  and the operator $L = - \Delta + c R_g - \Lambda$ applied on $| \nabla
  u|^{\tfrac{1}{2}}$ is given by
\begin{align}\begin{split}
& L (| \nabla u|^{\tfrac{1}{2}}) \\
= & - (| \nabla u|^{\tfrac{1}{2}})'' - \tfrac{2 \phi'}{\phi} (| \nabla
u|^{\tfrac{1}{2}})' \\
& + c (2 \phi^{- 2} K - 2 (\tfrac{2 \phi'}{\phi})' - \tfrac{3}{2}
(\tfrac{2 \phi'}{\phi})^2) | \nabla u|^{\tfrac{1}{2}} - \Lambda | \nabla
u|^{\tfrac{1}{2}} \\
= & 2 c (K - 1) \phi^{- 2} | \nabla u|^{\tfrac{1}{2}} .
\end{split}\end{align}
  Again, we trace back the Proof of Theorem \ref{final integral estimate}, we
  have that $\Lambda_c = \Lambda$ and that the equality holds in
  \[ \int_M | \nabla | \nabla u|^{\tfrac{1}{2}} |^2 + c R_g | \nabla u|
     \geq \Lambda_c \int_M | \nabla u| . \]
  We have already shown that $| \nabla u|^{\tfrac{1}{2}} \in H^1 (M)$, by
  definition of $\Lambda_c$ in \eqref{L c}, it follows that $L (| \nabla
  u|^{\tfrac{1}{2}}) = 0$ which gives that $K$ must be identically 1. Hence
  $(\mathbb{S}^2, g_0)$ is a standard 2-sphere. Now we conclude the proof of
  Theorem \ref{harmonic main}.
\end{proof}

\subsection{Model metrics of further generalizations}\label{further}

We give a description of the model metric and related functions
in the case of $\Lambda_{c,\mu}$ corresponding to Theorem \ref{thm1-mu} in dimension 3.

Let $M = (0,t_0)\times \mathbb{S}^2 $ with a warped metric $g_1 =
 \mathrm{d} t^2+\phi (t)^2 g_{\mathbb{S}^2} $ where $g_{\mathbb{S}^2}$ is the
standard round metric on the 2-sphere.
Let $\nabla$ be the Levi-Civita connection, $\Delta$ be the Laplacian with
respect to the metric $g_1$.
The functions $\mu$, $u$, $f$ only depend on $t$ and
satisfy the following relations:
\begin{enumerate}
  \item $f = -\tfrac{2(6-c^{-1})}{3(4-c^{-1})}\phi'/\phi;$
  \item $\Delta u + 3 f | \nabla u| = 0$;
  \item $c^{- 1} \mu + \tfrac{9 (8 - c^{- 1})}{2 (3 - c^{- 1})} f^2 - \tfrac{3
  (8 - c^{- 1})}{6 - c^{- 1}} f' = 2 \phi^{- 2}$;
  \item $\nabla | \nabla u|^{\tfrac{1}{2}} + \tfrac{3}{6 - c^{- 1}} f | \nabla
  u|^{-\tfrac{1}{2}} \nabla u = 0$;
   \item \(
  - \Delta | \nabla u|^{\tfrac{1}{2}} + c R_g | \nabla u|^{\tfrac{1}{2}} = \mu |
  \nabla u|^{\tfrac{1}{2}} 
\);
  \item $|\nabla u |^{1/2}$ is a constant multiple of $ \phi^{2/(4-c^{-1})}$.
\end{enumerate}
One can use the formulas in Appendix \ref{App} to check that the last three items
can actually be inferred from the first three. 

\appendix

\section{Details of calculations}\label{App}

In this appendix, we discuss a few properties of the metrics
\(g_1= \mathrm{d}t^2 + \phi(t)^2 g_{\mathbb{S}^{n-1}}\) related to $\Lambda_c$ and $\Lambda_{c,\mu}$ where $\phi$ satisfies the condition that $(\log\phi)''$ is strictly negative.

\begin{lemma}
Let $f$ and $\mu$ satisfy \eqref{f in mu} and \eqref{f mu relation}.
Then $v=\phi^{2/(4-c^{-1})}$ satisfies
\[-\Delta_{g_1} v + c R_{g_1}v=\mu v .\]
\end{lemma}

\begin{proof}
  This is a tedious calculation. We denote $f=\kappa \phi'/\phi$. Letting $h=((n-1)(\log \phi)'$, the scalar curvature $R_{g_1}$ is given by the following
  \begin{align}
    R_{g_1} &=(n-1)(n-2)\phi^{-2} - \tfrac{n}{n-1}h^{2} -2h'\\
            &=(n-1)\phi^{-2}[(n-2) - (n-2) (\phi')^2 - 2\phi''\phi ].
    \end{align}
  Recall that
  \begin{align}
  \mu &= - 2 \beta_2 c (n - 1)  f' + \frac{c (n - 2)  (n - 1)}{\phi^2} - \beta_2 cf^2  (n - 1) n\\
    &=-c(n-1)\phi^{-2}(2\beta_2 \kappa \phi \phi'' + (\beta_2\kappa^2n-2\beta_2\kappa) (\phi')^2 - (n-2)).
    \end{align}
    So
    \[c R_{g_1} -\mu = c(n-1)\phi^{-2}\left[(2\beta_2 \kappa -2) \phi \phi'' + (\beta_2\kappa^2n-2\beta_2\kappa-n+2) (\phi')^2 \right].
    \]
    We denote $\lambda=2/(4-c^{-1})$, we have that
    \[v^{-1}\Delta_{g_1} v = \phi^{-2}(\lambda\phi\phi''+ (n-1+\lambda(\lambda-1))(\phi')^2.\]
    Comparing coefficients of $\phi''$ and $(\phi')^2$ in $c R_{g_1} -\mu$ and $v^{-1}\Delta v$,
    we arrive that two identities
    \begin{align}
      \lambda &= c(n-1)(2\beta_2\kappa-2), \\
      n-1+\lambda(\lambda-1) &= c(n-1)(\beta_2\kappa^2n-2\beta_2\kappa-n+2).
    \end{align}
   We only have to check the above by using the explicit values of $\beta_2$, $\kappa$ and $\lambda$, and that concludes our proof of the identity
\[-\Delta_{g_1} v + c R_{g_1}v=\mu v. \]
  \end{proof}
  \begin{remark}
Note that the spectral constant $\Lambda_{c,\mu}$ is the first eigenvalue of the elliptic operator $-   \Delta_{g_1} +cR_{g_1} -\mu$, and the corresponding eigenfunction is strictly positive (or negative) away from $t=0$ and $t=t_0$, which is unique up to a constant multiple. If we have found such a function $u$ which satisfies $-  \Delta_{g_1} u+c R_{g_1}u-\mu u=1$, then $\Lambda_{c,\mu}(g_1)=1$.
\end{remark}

Now we deal with the case with $\mu$ being a constant $\Lambda>0$. The metric $g_1$ now is the $g_0$ as in Theorem \ref{thm1}. Let $\phi = a^2 \sin^2(b\theta)$, we can find positive constants $a,b$ that satisfy 
\begin{equation}\label{A0}
  \beta_2 a_1^2-\beta_2 a_1 b-\frac{n(n-2)}{4a^2}=0\text{ and }\Lambda_{c}(g_0)=\Lambda,
  \end{equation}
where $  a_1=\frac{1}{2}\sqrt{\frac{\Lambda(4n-(n-1) c^{-1} )}{4c(n-1)-(n-2)} }$, $ \beta_2=\frac{n}{n-1}\frac{4(n-1)-(n-2) c^{-1} }{4n-(n-1) c^{-1} }$, $0\leq c^{-1} <4$, and
\begin{equation*}
  g_0=  d\theta\otimes d\theta+a^2\sin^2(b\theta)g_{\mathbb{S}^{n-1}}
\end{equation*}
on $(0,\frac{\pi}{b})\times\mathbb{S}^{n-1}$. Firstly, we want to solve the following equation 
\begin{equation}\label{A1}
  -   \Delta_{g_0} u+c R_{g_0} u=\Lambda u,
\end{equation}
by taking $u=\sin^\lambda(b\theta)$. The scalar curvature $R_{g_0}$ of $g_0$ is given by 
\begin{equation*}
  R_{g_0}=(n-1)\left(nb^2+(n-2)(a^{-2}-b^2)\frac{1}{\sin^2(b\theta)}\right),
\end{equation*}
and since $u$ only depends on $\theta$, so
\begin{equation*}
  \Delta_{g_0}u=\frac{1}{\sqrt{\det g_0}}\frac{\partial}{\partial \theta}(\sqrt{\det g_0}g^{\theta\theta}\frac{\partial}{\partial\theta}u)=u''+(n-1)b\frac{\cos(b\theta)}{\sin(b\theta)}u'.
\end{equation*}
Then the equation \eqref{A1} is equivalent to
\begin{multline}\label{A2}
  - \left(u''+(n-1)b\frac{\cos(b\theta)}{\sin(b\theta)}u'\right)\\+(n-1)c\left(nb^2+(n-2)(a^{-2}-b^2)\frac{1}{\sin^2(b\theta)}\right)u=\Lambda u.
\end{multline}
Since $$u'=\lambda b\cos(b\theta)\sin^{\lambda-1}(b\theta)$$ and $$u''=-\lambda^2b^2 \sin^\lambda(b\theta)+\lambda(\lambda-1)b^2\sin^{\lambda-2}(b\theta),$$ the equation \eqref{A2} becomes
\begin{multline*}
 - c^{-1}  \left( -((n-1)\lambda+\lambda^2)b^2\sin^\lambda(b\theta)+(\lambda+n-2)\lambda b^2\sin^{\lambda-2}(b\theta)\right)\\+(n-1)nb^2\sin^\lambda(b\theta)
 +(n-1)(n-2)(a^{-2}-b^2)\sin^{\lambda-2}(b\theta)=c^{-1}\Lambda \sin^\lambda(b\theta),
\end{multline*}
which is equivalent to 
\begin{align*}
\begin{split}
  &\left( c^{-1}  (\lambda+n-1)\lambda b^2+n(n-1)b^2-c^{-1}\Lambda\right)\sin^2(b\theta)\\
  &+\left(- c^{-1} (\lambda+n-2)\lambda b^2+(n-1)(n-2)(a^{-2}-b^2)\right)=0.
 \end{split}
\end{align*}
Combing with the first equation of \eqref{A0}, $a$ and $b$ satisfy 
\begin{align*}
\begin{split}
   \begin{cases}
 	 c^{-1}   (\lambda+n-1)\lambda b^2+n(n-1)b^2-c^{-1}\Lambda=0, \\
 - c^{-1} (\lambda+n-2)\lambda b^2+(n-1)(n-2)(a^{-2}-b^2)	=0,\\
  \beta_2 a_1^2-\beta_2 a_1 b-\frac{n(n-2)}{4a^2}=0.
 \end{cases}
 \end{split}
\end{align*}
One can solve that 
\begin{equation*}
  \lambda=\frac{2}{4- c^{-1} }
\end{equation*}
and 
\begin{equation}\label{A3}
 \begin{cases}
 	a=& \sqrt{\frac{(n-1)(n-2)(4n-(n-1) c^{-1})}{\Lambda(4(n-2)-(n-3) c^{-1} )}},\\
 	b=&\frac{\sqrt{\Lambda}(4- c^{-1} )}{\sqrt{(4n-(n-1) c^{-1})(4(n-1)-(n-2) c^{-1} )}}.
 \end{cases}
\end{equation}
 
Note that the spectral constant $\Lambda_{c}$ is the first eigenvalue of the elliptic operator $-\Delta_{g_0} u+cR_{g_0} u$, and the corresponding eigenfunction is strictly positive (or negative) away from $\theta=0$ and $\theta=\frac{\pi}{b}$, which is unique up to a constant multiple. If we have found such a function $u$ which satisfies $-  \Delta_{g_0} u+c R_{g_0}u=\Lambda u$, then $\Lambda$ must be the first eigenvalue (spectral constant). Hence the positive constants $a,b$ given in \eqref{A3} satisfy \eqref{A0}.

\bibliographystyle{alpha}
\bibliography{spectral-llarull}
\end{document}